\documentclass[a4paper]{amsart}

\usepackage{tikz, tikz-cd}
\usepackage{amsmath,amssymb}
\usepackage{amsthm}
\usepackage{mathrsfs}
\usepackage[shortlabels]{enumitem}
\usepackage{mathtools}
\usepackage[all,cmtip,2cell]{xy}
\usepackage{array}
\usepackage{comment}
\usepackage{braket}
\usepackage{stmaryrd}
\usepackage{graphicx}
\usepackage{url}
\usepackage{ascmac}

\theoremstyle{plain}
\newtheorem{thm}{Theorem}[section]
\newtheorem{prop}[thm]{Proposition}
\newtheorem{lem}[thm]{Lemma}
\newtheorem{claim}[thm]{Claim}
\newtheorem{cor}[thm]{Corollary}
\newtheorem*{thm*}{Theorem}

\theoremstyle{definition}

\newtheorem*{NaC}{Notation and Convention}
\newtheorem*{ACK}{Acknowledgment}

\theoremstyle{remark}
\newtheorem{rem}[thm]{Remark}

\numberwithin{equation}{section}

\newcommand{\Z}{\mathbb{Z}}
\newcommand{\Q}{\mathbb{Q}}

\newcommand{\C}{\mathbb{C}}
\renewcommand{\P}{\mathbb{P}}

\newcommand{\I}{\mathcal{I}}

\newcommand{\D}{\Delta}
\newcommand{\e}{\varepsilon}

\newcommand{\G}{\Gamma}
\newcommand{\s}{\sigma}

\newcommand{\mf}{\mathfrak}

\newcommand{\emp}{\varnothing}

\newcommand{\ol}{\overline}
\newcommand{\wt}{\widetilde}
\newcommand{\wh}{\widehat}

\newcommand{\ra}{\Rightarrow}

\newcommand{\hra}{\hookrightarrow}
\newcommand{\epm}{\twoheadrightarrow}

\DeclareMathOperator{\rk}{rk}

\DeclareMathOperator{\Sym}{\mathrm{Sym}}

\DeclareMathOperator{\Spec}{\mathrm{Spec}}

\DeclareMathOperator{\Proj}{\mathrm{Proj}}

\DeclareMathOperator{\Exc}{\mathrm{Exc}}

\DeclareMathOperator{\Pic}{\mathrm{Pic}}

\DeclareMathOperator{\red}{\mathrm{red}}

\DeclareMathOperator{\Bl}{\mathrm{Bl}}

\DeclareMathOperator{\Bs}{\mathrm{Bs}}

\DeclareMathOperator{\pr}{pr}

\DeclareMathOperator{\Hom}{Hom}

\DeclareMathOperator{\Ker}{Ker}
\DeclareMathOperator{\Cok}{Cok}

\DeclareMathOperator{\Ext}{Ext}

\DeclareMathOperator{\Hilb}{Hilb}

\newcommand{\mcC}{\mathcal{C}}
\newcommand{\mcE}{\mathcal{E}}
\newcommand{\mcF}{\mathcal{F}}
\newcommand{\mcG}{\mathcal{G}}
\newcommand{\mcO}{\mathcal{O}}
\newcommand{\mcI}{\mathcal{I}}
\newcommand{\mcL}{\mathcal{L}}
\newcommand{\mcN}{\mathcal{N}}
\newcommand{\mcM}{\mathcal{M}}

\let\Im\relax
\DeclareMathOperator{\Im}{\mathrm{Im}}

\title[\tiny{Rank $2$ weak Fano bundles on del Pezzo $3$-folds of degree $4$}]{Classification of rank two weak Fano bundles on del Pezzo threefolds of degree four}
\author[T.FUKUOKA, W.HARA, D.ISHIKAWA]{Takeru Fukuoka Wahei Hara, Daizo Ishikawa}

\address[T.FUKUOKA]{Graduate School of Mathematical Sciences\\The University of Tokyo\\3-8-1 Komaba\\Meguro-ku, Tokyo 153-8914, Japan}
\email{tfukuoka@ms.u-tokyo.ac.jp}

\address[W.HARA]{The Mathematics and Statistics Building, University of Glasgow, University Place, Glasgow, G12 8QQ, UK.}
\email{wahei.hara@glasgow.ac.uk}

\address[D.ISHIKAWA]{Department of Mathematics, School of Science and Engineering, Waseda University, Ohkubo 3-4-1, Shinjuku, Tokyo 169-8555, JAPAN}
\email{azoth@toki.waseda.jp}

\date{\today}



\begin{document}
\maketitle
\begin{abstract}
We classify rank two vector bundles on a given del Pezzo threefold of degree four whose projectivizations are weak Fano into seven cases. 
We also give an example for each of these seven cases.
\end{abstract}
\tableofcontents
\section{Introduction}

\subsection{Motivation}
The motivation of this study comes from the classification of Fano manifolds. 
Since smooth Fano $3$-folds were classified (see \cite{Fanobook} and references therein), 
many researchers have treated the classification of Fano $4$-fold having projective bundle structures. 
In 1990, Szurek and Wi\'{s}niewski called a vector bundle whose projectivization is Fano a \emph{Fano bundle} and gave a classification of rank $2$ Fano bundles on $\P^{3}$ or a smooth hyperquadric $\Q^{3}$ \cite{sw}. 
After that, the classification of rank $2$ Fano bundles over smooth Fano $3$-folds 
has been addressed by many researchers (e.g. \cite{sw,lan,mos1,mos2}). 
In particular, Mu\~{n}oz, Occhetta, and Sol\'{a} Conde classified rank $2$ Fano bundles on Fano $3$-folds of Picard rank $1$ \cite{mos2}. 

On the other hand, after the establishment of the classification of Fano $3$-folds, 
several papers also started classifying weak Fano $3$-folds. 
In 1989, K. Takeuchi developed his 2-ray game method by considering not only Fano $3$-folds but also weak Fano $3$-folds of Picard rank $2$ \cite{Takeuchi89}. 
By this $2$-ray game method, he successfully gave a concise proof of the existence of a line on a Mukai $3$-fold, which is a new perspective on the classification of Fano $3$-folds. 
Since the establishment of the $2$-ray game method, classifying weak Fano $3$-folds of Picard rank $2$ has been considered to be significant and treated by many researchers (e.g. \cite{Takeuchi89,lan,JPR05,CM13}).

In view of these previous researches, classifying weak Fano $4$-folds with Picard rank $2$ would be important to investigate Fano $4$-folds. 
Our approach to this problem is to consider weak Fano $4$-folds with $\P^{1}$-bundle structures, as Szurek-Wi\'{s}niewski did \cite{sw}. 

\subsection{Known classification of weak Fano bundles}
In this paper, we say that a vector bundle $\mcE$ on a smooth projective variety $X$ is \emph{weak Fano} if its projectivization $\P_{X}(\mcE)$ is a weak Fano manifold. 
Weak Fano bundles are firstly introduced by Langer \cite{lan} as generalizations of Fano bundles. 
Until now, rank $2$ weak Fano bundles are classified when the base space is the projective space \cite{lan,yas,Ohno16v7} or a hypercubic in $\P^{4}$ \cite{ishi}. 

We quickly review these known results. 
On $\P^{2}$, Langer \cite{lan} and Ohno \cite{Ohno16v7} classified weak Fano bundles whose 1st Chern class is odd. 
Yasutake \cite{yas} classified rank $2$ weak Fano bundles over $\P^{2}$ 
whose 1st Chern class is even 
except for the existence of a weak Fano rank $2$ bundle $\mcE$ with $c_{1}(\mcE)=0$ and $c_{2}(\mcE)=6$. 
Meanwhile, Cutrone-Marshburn \cite[P.2, L.14]{CM13} pointed out that such a weak Fano bundle exists as the Bordiga scroll \cite{OttavianiBordiga}. 
Yasutake also classified rank $2$ weak Fano bundles on $\P^{n}$ with $n \geq 3$. 
Thus the rank $2$ weak Fano bundles over projective spaces are classified. 
Rank $2$ weak Fano bundles on smooth cubic hypersurfaces in $\P^{4}$ are also classified by the 3rd author of this article \cite{ishi}.

\subsection{Instanton bundles and Ulrich bundles}\label{sec-instantonUlrich}
Besides, weak Fano bundles have emerged in the study of different contexts. 
For example, minimal instanton bundles and special Ulrich bundles are most studied weak Fano (but not Fano) bundles on del Pezzo $3$-folds. 

Let $X$ be $\P^{3}$ or a del Pezzo $3$-fold of Picard rank $1$, i.e., a smooth Fano $3$-fold of Picard rank $1$ whose canonical divisor is divisible by $2$. 
Then an \emph{instanton bundle} on $X$ is defined to be 
a rank $2$ slope stable vector bundle $\mcE$ with $c_{1}(\mcE)=0$ and $H^{1}(\mcE(\frac{K_{X}}{2}))=0$ \cite{ADHM,Kuznetsov2012,Faenzi2014}. 
The moduli space of instanton bundles on $\P^{3}$ is a significant object for mathematical physics. 
Indeed, a certain subset of this moduli space corresponds to self-dual solutions of the $\mathrm{SU}(2)$ Yang-Mills equations on the $4$-sphere $S^{4}$ up to gauge equivalence \cite{Atiyah-Ward,ADHM}. 
After this, 
Faenzi \cite{Faenzi2014} and Kuznetsov \cite{Kuznetsov2012} defined a generalized notion for instanton bundles $\mcE$ on a del Pezzo $3$-fold $X$ as above. 
By the above reason, moduli spaces of instanton bundles on $\P^{3}$ and del Pezzo $3$-folds have been studied deeply (e.g. \cite{Faenzi2014,Kuznetsov2012}). 

%
%
For an instanton bundle $\mcE$ on $X$, 
it is known that $-K_{X}.c_{2}(\mcE) \geq 4$ 
(cf. \cite[Corollary~3.2]{Kuznetsov2012}, \cite[Lemma~1.2]{Faenzi2014}). 
Thus an instanton bundle $\mcE$ is said to be \emph{minimal} if $-K_{X}.c_{2}(\mcE)=4$. 
When $X$ is a del Pezzo $3$-fold, 
a vector bundle $\mcE$ is a minimal instanton bundle if and only if $\mcE(1)$ is a \emph{special Ulrich} bundle of rank $2$ \cite[Lemma~3.1]{Kuznetsov2012}, 
which is defined to be a rank $2$ vector bundle $\mcF$ 
such that $\det \mcF \simeq \mcO_{X}(-K_{X}) \simeq \mcO_{X}(2)$ and 
$H^{\bullet}(\mcF(-j))=0$ for every $j \in \{1,2,3\}$. 
When $d:=\mcO_{X}(1)^{3} \geq 3$, the above definition is equivalent to say that $\mcF$ has the following linear resolution on $\P^{d+1}$: 
\[0 \to \mcO_{\P^{d+1}}(2-d)^{\oplus b_{d-2}} \to \mcO_{\P^{d+1}}(3-d)^{\oplus b_{d-3}} \to \cdots \to \mcO_{\P^{d+1}}^{\oplus b_{0}} \to \mcF \to 0,\]
where $b_{i}=2d\binom{d-2}{i}$ \cite{Beauville}. 
Moreover, Beauville showed that 
every special Ulrich bundle $\mcF$ of rank $2$ is isomorphic to 
a unique non-trivial extension of $\I_{C}(2)$ by $\mcO_{X}$, 
where $C$ is a normal elliptic curve in $X$ \cite[Remark~6.3]{Beauville}. 
Since every normal elliptic curve $C$ is defined by quadratic equations \cite{Hulek}, 
$\mcF$ is globally generated. 
Since $\det \mcF \simeq \mcO(-K_{X})$, 
we conclude that $\mcF$ is a weak Fano bundle. 

In summary, on a del Pezzo $3$-fold of degree $d \geq 3$, 
minimal instanton bundles are the most well-studied examples of weak Fano bundles. 
In this article, we supplementally show that every slope stable weak Fano bundle $\mcE$ with $\rk\mcE=2$ and $c_{1}(\mcE)=0$ is an instanton bundle (see Corollary~\ref{cor-instanton}). 

\subsection{Main Result}
In this article, we classify rank $2$ weak Fano bundles over a del Pezzo $3$-fold of degree $4$, which is a smooth complete intersections of two hyperquadrics in $\P^{5}$ \cite{Fujita80}.  
Our classification is given as follows. 

\begin{thm} \label{mthm}
Let $X$ be a smooth complete intersection of two hyperquadrics in $\P^{5}$. 
For a normalized bundle $\mcE$ of rank $2$, 
$\mcE$ is a weak Fano bundle if and only if $\mcE$ is one of the following. 
\renewcommand{\labelenumi}{(\roman{enumi})}
\begin{enumerate}
\item $\mcO_X\oplus\mcO_X(-1)$. 
\item $\mcO_X\oplus\mcO_X$. 
\item $\mcO_X(1)\oplus\mcO_X(-1)$. 
\item A unique non-trivial extension of $\I_{L}$ by $\mcO_X$, 
where $L$ is a line in $X$.
\item A unique non-trivial extension of $\I_{C}$ by $\mcO_X(-1)$, 
where $C$ is a smooth conic in $X$.
\item A unique non-trivial extension of $\I_{C}(1)$ by $\mcO_X(-1)$, 
where $C$ is a non-degenerate smooth elliptic curve in $X$ of degree $6$. 
\item A unique non-trivial extension of $\I_{C}(1)$ by $\mcO_X(-1)$, 
where $C$ is a non-degenerate smooth elliptic curve in $X$ of degree $7$ which is defined by quadratic equations. 
\end{enumerate}
In the above statement, a unique non-trivial extension $\mcE$ of $\mcF$ by $\mcG$ means that $\dim \Ext^{1}(\mcF,\mcG)=1$ and $\mcE$ fits into an exact sequence $0 \to \mcG \to \mcE \to \mcF \to 0$ such that the extension class is non-zero in $\Ext^{1}(\mcF,\mcG)$. 

Furthermore, on an arbitrary smooth complete intersection of two hyperquadrics in $\P^{5}$, 
there exist examples for each case of (i) -- (vii). 
\end{thm}

\begin{rem}\label{rem-main}
We give some remarks about the above result. 
\begin{enumerate}
\item The vector bundles of type (i), (ii) and (v) are Fano bundles \cite{mos2}. Others are not Fano bundles but weak Fano bundles. 
\item The vector bundles of type (i-iv) are not slope stable. Others are slope stable. 
\item For every vector bundle $\mcE$, 
$\mcE$ is of type (v) 
if and only if 
$\mcE$ is the pull-back of the Spinor bundle on a hyperquadric $Q$ under a finite morphism $X \to Q$ (see Theorem~\ref{thm-H}). 
In particular, the vector bundles of type (v) are Fano bundles \cite{mos2}. 
\item For every vector bundle $\mcE$, 
$\mcE$ is of type (vi) 
if and only if 
$\mcE(1)$ is a special Ulrich bundle \cite[Remark~6.3.3]{Beauville}, 
which is equivalent to say that 
$\mcE$ is an instanton bundle with $c_{2}(\mcE)=2$ as in explained in $\S$~\ref{sec-instantonUlrich}. 
\item Every vector bundle $\mcE$ of type (vii) is an instanton bundle with $c_{2}(\mcE)=3$ (see Corollary~\ref{cor-instanton}). 
Unlike the case (3), there is an instanton bundle $\mcE$ with $c_{2}(\mcE)=3$ which is not a weak Fano bundle (see Remark~\ref{rem-nonweakfano}). 
\end{enumerate}
\end{rem}

Our result is essentially different from the classification of rank $2$ weak Fano bundles on a hypercubic in $\P^{4}$ \cite{ishi}, which was done by the 3rd author of this article. 
Briefly speaking, 
he proved that slope stable rank $2$ weak Fano bundles are always minimal instanton \cite[Theorem~1.1]{ishi}. 
In contrast, on a complete intersection of two hyperquadrics, 
other stable bundles appear. 
Indeed, every weak Fano bundle $\mcE$ of type (vii) in Theorem~\ref{mthm} is 
a non-minimal instanton bundle. 
Our result also shows that the zero scheme of a general section of $\mcE(1)$ is a non-projectively normal elliptic curve $C$ defined by quadratic equations. 
Hence $\mcE(1)$ is not $0$-regular but globally generated. 
Note that we can not remove the condition that $C$ is defined by quadratic equations (cf. Theorem~\ref{thm-F-chara} and Remark~\ref{rem-nonweakfano}). 

\subsection{Key results for proving Theorem~\ref{mthm}}

Here we collect the key ingredients for the proof of Theorem~\ref{mthm}. 

Let $X$ be a smooth complete intersection of two hyperquadrics in $\P^{5}$. 
Let $H_{X}$ be a hyperplane section of $X$ and $L_{X}$ a line on $X$. 
Then $\Pic(X) \simeq H^{2}(X,\Z) \simeq \Z[H_{X}]$ and $H^{4}(X,\Z) \simeq \Z[L_{X}]$. 
Using the above isomorphisms, we identify the cohomology classes with integers. 
Note that the property to be weak Fano is slope stable under taking tensor products with line bundles. 
To classify rank $2$ weak Fano bundles $\mcE$, 
we may assume that $\mcE$ is \emph{normalized}, 
i.e., $c_{1}(\mcE) \in \{0,-1\}$.  

Firstly, we will characterize normalized rank $2$ weak Fano bundles having global sections as follows.

\begin{thm}\label{thm-I1}
Let $X$ be a del Pezzo $3$-fold of degree $4$. 
Let $\mcE$ be a normalized rank $2$ weak Fano bundle on $X$. 
Then we obtain the following assertions. 
\begin{enumerate}
\item $h^{0}(\mcE)>0$ if and only if $\mcE$ is one of (i), (ii), (iii), or (iv) in Theorem~\ref{mthm}. 
\item $h^{0}(\mcE)=0$ if and only if $(c_{1}(\mcE),c_{2}(\mcE))=(-1,2),(0,2),(0,3)$. 
\end{enumerate}
\end{thm} 

By this theorem, it suffices to treat weak Fano bundles $\mcE$ with $(c_{1}(\mcE),c_{2}(\mcE))=(-1,2),(0,2),(0,3)$. 

Secondaly, we characterize weak Fano bundles $\mcE$ with $(c_{1}(\mcE),c_{2}(\mcE))=(-1,2)$ as follows. 

\begin{thm}\label{thm-H}
Let $X$ be a del Pezzo $3$-fold of degree $4$. 
Let $\mcE$ be a rank $2$ vector bundle on $X$. 
Then the following conditions are equivalent.
\begin{enumerate}
\item $\mcE$ is a weak Fano bundle with $(c_{1}(\mcE),c_{2}(\mcE))=(-1,2)$. 
\item $\mcE$ is a Fano bundle with $(c_{1}(\mcE),c_{2}(\mcE))=(-1,2)$. 
In particular, by the classification of Fano bundles on a del Pezzo $3$-fold of degree $4$ \cite{mos2}, $\mcE$ is the restriction of the Spinor bundle \cite{Ottaviani88} on a smooth hyperquadric in $\P^{5}$ containing $X$, or the pull-back of the Spinor bundle on the $3$-dimensional smooth hyperquadric $\Q^{3}$ under a double covering $X \to \Q^{3}$. 
\item $\mcE$ is (v) in Theorem~\ref{mthm}. 
\end{enumerate}
In particular, there exist examples for (v) in Theorem~\ref{mthm}. 
\end{thm}

Thirdly, we characterize weak Fano bundles $\mcE$ with $(c_{1}(\mcE),c_{2}(\mcE))=(0,2)$ or $(0,3)$ as follows. 
\begin{thm}\label{thm-F-chara}
Let $X$ be a del Pezzo $3$-fold of degree $4$. 
Let $\mcE$ be a rank $2$ vector bundle on $X$. 
Then the following assertions are equivalent. 
\begin{enumerate}
\item $\mcE$ is weak Fano with $(c_{1}(\mcE),c_{2}(\mcE))=(0,2)$ (resp. $(0,3)$). 
\item $\mcE$ is (vi) (resp. (vii)) in Theorem~\ref{mthm}. 
\end{enumerate}
\end{thm}

Then by Theorems~\ref{thm-I1}, \ref{thm-H}, and \ref{thm-F-chara}, we obtain the characterization part of Theorem~\ref{mthm}. 

Finally, we show the existence of $\mcE$ for each condition in Theorem~\ref{mthm}, which is namely the following theorem. 
\begin{thm}\label{thm-existence}
For an arbitrary del Pezzo $3$-fold of degree $4$, there exist examples for each one of (i) -- (vii) in Theorem~\ref{mthm}. 
\end{thm}

Then Theorem~\ref{mthm} follows from Theorems~\ref{thm-I1}, \ref{thm-H}, \ref{thm-F-chara}, and \ref{thm-existence}.

\subsection{Outline of proof of our results}

\subsubsection{Outline for Theorem~\ref{thm-I1}}

Let $X$ be a del Pezzo $3$-fold of degree $4$. 
Theorem~\ref{thm-I1}~(1) is a characterization of weak Fano bundles of rank $2$ on $X$ which is not slope semi-stable (cf. \cite[Lemma~3.2]{ishi}). 
For proving Theorem~\ref{thm-I1}~(2), 
we will bound the second Chern classes of such an $\mcE$ by using the inequality $(-K_{\P_{X}(\mcE)})^{4}>0$ and the Le Potier vanishing (cf. \cite[Theorem~7.3.5]{laz2}). 
By this argument, we characterize slope stable weak Fano bundles of rank $2$ in terms of its Chern classes.

\subsubsection{Outline for Theorems~\ref{thm-H} and \ref{thm-F-chara}}

The most non-trivial part of Theorems~\ref{thm-H} and \ref{thm-F-chara} is 
to show the global generation of $\mcE(1)$
for a rank $2$ weak Fano bundle $\mcE$ with $(c_{1},c_{2}) \in \{(-1,2),(0,2),(0,3)\}$. 
Once we prove the global generation, 
we can produce a smooth curve $C$ 
as the zero scheme of a general global section of $\mcE(1)$ 
and hence $\mcE$ can be written as a unique extension as in (v), (vi), (vii) in Theorem~\ref{mthm}. 

To show Theorem~\ref{thm-H}, we will use the fact that $h^{0}(\mcE(1))>0$ for every normalized weak Fano bundle $\mcE$ of rank $2$ (=Proposition~\ref{bounds}~(1)). 
When $(c_{1},c_{2})=(-1,2)$, the zero scheme of a global section of $\mcE(1)$ is a conic and hence we can show that $\mcE(1)$ is globally generated. 
Moreover, this fact implies that $\mcE$ is a Fano bundle, which is already classified in \cite{mos2}. 

To show Theorem~\ref{thm-F-chara}, 
we will prove the following theorem.
\begin{thm}\label{F-thm-glgen}
Let $X$ be a smooth Fano $3$-fold of Picard rank $1$. 
Let $\mcF$ be a nef vector bundle with $\det \mcF=\mcO(-K_{X})$ and $c_{1}(\mcF)^{3}-2c_{1}c_{2}(\mcF)+c_{3}(\mcF) > 0$. 

If $\mcF$ is not globally generated, then $X$ is a del Pezzo $3$-fold of degree $1$ and $\mcF \simeq \mcO_{X}^{\oplus \rk \mcF-2} \oplus \mcO_{X}(\frac{-K_{X}}{2})^{\oplus 2}$. 
\end{thm}
Applying this result, we see that $\mcE(1)$ is globally generated for every weak Fano bundle $\mcE$ with $c_{1}(\mcE)=0$ on a del Pezzo $3$-fold of degree $4$, 
which will imply Theorem~\ref{thm-F-chara}. 
As another corollary of Theorem~\ref{F-thm-glgen}, we will also see that every weak Fano bundle $\mcE$ with $\det \mcE=\mcO_{X}$ on a del Pezzo $3$-fold $X$ of Picard rank $1$ is an instanton bundle (=Corollary~\ref{cor-instanton}). 

Our proof of Theorem~\ref{F-thm-glgen} goes as follows. 
Let $X$ be a Fano $3$-fold of Picard rank $1$ and $\mcF$ a vector bundle with $c_{1}(\mcF)=c_{1}(X)$. 
Let $\pi \colon M:=\P_{X}(\mcF) \to X$ be the projectivization and $\xi$ a tautological divisor. 
Assume that $\Bs|\xi| \neq \emp$ and $\xi$ is nef and big. 
Since $-K_{M}=(\rk \mcF)\xi$, $M$ is a weak Mukai manifold. 
Then by using some results about ladders of Mukai varieties \cite{Shokurov,Isk79,Reid,Shin,Mella,Minagawa}, 
we can obtain a smooth K3 surface $S \subset M$ by taking general $(\rk \mcF)$ elements of $|\xi|$ (=Theorem~\ref{thm-Mella}). 
Since the complete linear system $|\xi|_{S}|$ still has a non-empty base locus, 
there is an elliptic curve $B$ and a smooth rational curve $\Gamma$ on $S$ such that $\xi|_{S} \sim gB+\Gamma$, where $g=(1/2)\xi^{\rk\mcF+2}+1 > 0$ \cite{SD74,Shokurov}. 
Note that $\Bs|\xi|=\Gamma$. 

Let us suppose that $\Gamma$ is contracted by $\pi \colon M \to X$. 
Then it is easy to see that $\Gamma$ is a line in a fiber of $\pi$. 
As in \cite[$\S$~6]{Isk79}, we consider the blowing-up of $M$ along $\Gamma$. 
Since $\Gamma$ is a line of a $\pi$-fiber, 
we can compute the conormal bundle of $\Gamma$ and hence conclude that the anti-canonical image of $\Bl_{\Gamma}M$ is the join of $\P^{1} \times \P^{2}$ and some linear space. 
This structure enables us to show $X$ is a del Pezzo $3$-fold of degree $1$ and $\mcF \simeq \mcO_{X}(\frac{-K_{X}}{2})^{\oplus 2} \oplus \mcO_{X}^{\oplus \rk\mcF-2}$. 
The most technical part of this proof is to show that $\Gamma$ is contracted by $\pi \colon M \to X$. 
For more precise, see $\S$~\ref{subsubsec-GammaContracted}.

\subsubsection{Outline for Theorem~\ref{thm-existence}}

Fix an arbitrary smooth complete intersection $X$ of two hyperquadrics in $\P^{5}$. 
The existence of (i), (ii) and (iii) in Theorem~\ref{mthm} is obvious. 
For the existence of (iv), (v), (vi), and (vii) in Theorem~\ref{mthm}, it suffices to find a line, a conic, a non-degenerate sextic elliptic curve, and a non-degenerate septic elliptic curve defined by quadratic equations respectively. 
Indeed, if such a curve is found, the Hartshorne-Serre correspondence gives a vector bundle which we want to find. 
The existence of lines and conics are obvious since a general hyperplane section, which is a del Pezzo surface, contains such curves. 
The existence of a non-degenerate sextic elliptic curve was already proved by \cite[Lemma~6.2]{Beauville}. 
Note that such a curve is always defined by quadratic equations (cf. \cite{Hulek}). 

The remaining case is the existence of a septic elliptic curve defined by quadratic equations. 
Note that there is a non-degenerate septic elliptic curve in $\P^{5}$ defined by quadratic equations (cf. \cite[Theorem~3.3]{Crauder-Katz}, \cite[$\S$~2]{Hulek-Katz-Schreyer}). 
Hence an example of a del Pezzo $3$-fold of degree $4$ containing such a curve was already known. 
In this article, we will slightly improve this known result by showing that \emph{every} del Pezzo $3$-fold of degree $4$ contains such a curve.

To obtain this result, we study a non-degenerate septic elliptic curve $C$ in $\P^{5}$ and 
show that $C$ is defined by quadratic equations if and only if 
$C$ has no trisecants (=Proposition~\ref{7-stability}). 
For this characterization, we will use Mukai's technique \cite{Mukai}, which interprets the property of being defined by quadratic equations into the vanishing of the 1st cohomologies of a certain family of vector bundles on $C$. 
Then as in \cite[Lemma~6.2]{Beauville} and \cite{Hartshorne-Hirschowitz}, we will produce such a curve $C$ by smoothing the union of a conic and an elliptic curve of degree $5$, which has no trisecants. 
For more precise, see $\S$~\ref{subsec-Smoothing}.

\subsection{Organization of this article}
In $\S$~2, we give certain numerical conditions for rank $2$ weak Fano bundles over del Pezzo $3$-folds of degree $4$ and prove Theorem~\ref{thm-I1}. 
In $\S$~3, we prove Theorem~\ref{thm-H}. 
In $\S$~4, we prove Theorem~\ref{F-thm-glgen} and Theorem~\ref{thm-F-chara} as its corollary. 
In $\S$~5, we find a septic elliptic curve $C$ on an arbitrary del Pezzo $3$-fold of degree $4$ and prove Theorem~\ref{thm-existence}.

\begin{NaC}
Throughout this article, we will work over the complex number field. 
We regard vector bundles as locally free sheaves. 
For a vector bundle $\mcE$ on a smooth projective variety $X$, 
we define $\P_{X}(\mcE):=\Proj \Sym \mcE$. 
In this article, we only deal with slope stability among the stabilities. 
For a given ample divisor $H$, we say that $\mcE$ is $H$-stable (resp. $H$-semi-stable) if $\frac{c_{1}(\mcF)}{\rk \mcF}.H^{\dim X-1} < (\text{resp. }\leq) \frac{c_{1}(\mcE)}{\rk \mcE}.H^{\dim X-1}$ for every subsheaf $0 \neq \mcF \subsetneq \mcE$. 
When $X$ is of Picard rank $1$, we abbreviate $H$-stable (resp. $H$-semi-stable) to simply stable (resp. semi-stable). 

We say that $\mcE$ is \emph{weak Fano} if $\P_{X}(\mcE)$ is a weak Fano manifold, that is, $-K_{\P_{X}(\mcE)}$ is nef and big. 
When $X$ is a smooth Fano $3$-fold of Picard rank $1$, 
we often identify $H^{2i}(X,\Z) \simeq \Z$ by taking an effective generator. 
We also say a rank $2$ vector bundle $\mcE$ on $X$ is \emph{normalized} when $c_{1}(\mcE) \in \{0,-1\}$. 
\end{NaC}

\begin{ACK}
The first author would like to show his gratitude to Professors Yoshinori Gongyo and Hiromichi Takagi for valuable discussions and warm encouragement. 
The second and third authors are grateful to Professors Hajime Kaji and Yasunari Nagai for fruitful discussions, helpful comments, and their encouragement. 
We are grateful to Doctor Akihiro Kanemitsu for his suggestions and comments. 
We also are grateful to the referee for careful reading and fruitful comments. 
The first author was partially supported by JSPS Research Fellowship for Young Scientists (JSPS KAKENHI Grant Number 18J12949). 
The second author was supported by JSPS Research Fellowship for Young Scientists (JSPS KAKENHI Grant Number 17J00857). 
\end{ACK}

\section{Numerical bounds: Proof of Theorem~\ref{thm-I1}}\label{sec-I1}
In this section, $X$ denotes a smooth del Pezzo $3$-fold of degree $4$ and $H_{X}$ a class of a hyperplane section. 
Let us recall that $H^{2}(X,\Z)$ is generated by $H_{X}$ and $H^{4}(X,\Z)$ is generated by $H_{X}^{2}/4$. 
If we naturally identify $H^{2}(X,\Z)$ as $\Z$ and $H^{4}(X,\Z)$ as $\Z$ by these generators, 
we can take integers $c_{1},c_{2}$ as $c_{1}(\mcE)=c_{1} \cdot H_{X}$ and $c_{2}(\mcE)=c_{2} \cdot H_{X}^{2}/4$ for a given vector bundle $\mcE$. 
We also call this integer $c_{i}$ the $i$-th Chern class of $\mcE$ for each $i \in \{1,2\}$. 
Let $\pi \colon \P_{X}(\mcE) \to X$ be the projectivization, $\xi$ the tautological line bundle, and $H:=\pi^{\ast}H_{X}$. 
In this setting, we have $-K_{\P_{X}(\mcE)}=2\xi+(2-c_{1})H$, $H^{4}=0$, and $\xi.H^{3}=4$.

The goal of this section is to prove Theorem~\ref{thm-I1}. 
First of all, it is easy to check that the vector bundle $\mcE=\mcO_{X} \oplus \mcO_{X}(a)$ with $a \in \Z_{\geq 0}$ is weak Fano if and only if $a \in \{0,1,2\}$ since $\xi$ and $H$ generates the nef cone of $\P_{X}(\mcE)$ and $-K_{\P_{X}(\mcE)}=2\xi+(2-a)H$. 
This means that, for a rank $2$ normalized bundle $\mcE$ on $X$, the conditions (i-iii) in Theorem~\ref{mthm} are equivalent to $\mcE$ being a decomposable weak Fano bundle. 
Thus, to prove Theorem~\ref{thm-I1} for a given rank $2$ weak Fano normalized bundle $\mcE$ on $X$, 
we will characterize $\mcE$ with global sections as the condition (iv) in Theorem~\ref{mthm}, and compute the Chern classes of $\mcE$ with no global sections. 

\begin{lem}\label{lem-comp}
Let $X$ be a del Pezzo $3$-fold of degree $4$ and $\mcE$ a rank $2$ weak Fano normalized bundle with Chern classes $c_{1},c_{2} \in \Z$. 

Then we obtain the following assertions. 
\begin{enumerate}
\item For any integer $j$, 
\begin{align}\label{eq-RR}
\chi( \mcE (j)) = 
\left\{
\begin{array}{ll}
\frac{4}{3}j^3+2j^2+\frac{16-6c_{2}}{6}j-\frac{1}{2}c_{2}+1
& \text{ if } c_{1} = -1 \text{ and } \\
\frac{4}{3}j^{3}+ 4j^{2} + \frac{14-3c_{2}}{3}j -c_{2}+2
& \text{ if } c_{1} = 0. 
\end{array}
\right.
\end{align}\label{lem-comp-RR}
\item When $c_1=-1$, we have ${\xi}^{2}.H^{2}=-4$, ${\xi}^{3}.H=-c_2+4$, and ${\xi}^{4}=2c_{2}-4$. 
When $c_{1}=0$, we have ${\xi}^{2}.H^{2}=0$, ${\xi}^{3}.H=-c_{2}$, and ${\xi}^{4}=0$. 
\label{lem-comp-IN}
\item For any $i \geq 2$ and $j \geq 0$, we have $H^{i}(X,\mcE (j)) =0$. \label{lem-comp-LP}
\end{enumerate}
\end{lem}
\begin{proof}
(1) follows from the Hirzebruch-Riemann-Roch formula. 
(2) follows from the relation ${\xi}^{2}-c_1\xi.H+\frac{c_2}{4}H^{2}=0$ (cf.~\cite[Appendix A]{ha})
and the fact that $H^{4}=0$ and $\xi.H^3=4$ on $\P_{X}(\mcE)$. 
Let us prove (3). 
Since $-K_{\P_{X}(\mcE)}=2\xi + (2-c_{1})H$ is nef and big, $\mcE (2)$ is ample. 
Thus the Le Potier vanishing theorem \cite[Theorem~7.3.5]{laz2}
gives $H^{i}(X,\mcE (j)) = H^{i}(X,\omega_{X} \otimes \mcE (2+j)) = 0$ for any $i \geq 2$ and $j \geq 0$. 
We complete the proof. 
\end{proof}


Following \cite[Section~1.1]{mos1}, 
we introduce the invariant $\beta:=\min\{n \in \Z \mid h^{0}(\mcE(n))>0\}$ for a given vector bundle $\mcE$. 
Then $\mcE$ has $\mcO(-\beta)$ as a subsheaf, which is saturated by the definition of $\beta$. 
Hence we have the following exact sequence
\begin{align}\label{ex-beta}
0 \to \mcO(-\beta) \to \mcE \to \mcI_{Z}(c_{1}+\beta) \to 0,
\end{align}
where $Z$ is a closed subscheme of $X$ of purely codimension $2$ or the empty set. 
This exact sequence (\ref{ex-beta}) means that $c_{2}(\mcE(\beta)) = [Z]$ belongs an effective class. 

Now we bound $c_{2}$ and $\beta$ for a normalized weak Fano bundle of rank $2$ as follows.

\begin{prop}\label{bounds}
Let $X$ be a del Pezzo $3$-fold of degree $4$. 
Let $\mcE$ be a rank $2$ normalized weak Fano bundle on $X$ with Chern classes $c_1$ and $c_2$. 
Set $\beta:=\min\{n \in \Z \mid H^{0}(X,\mcE(n)) \neq 0\}$. 
Then we obtain the following assertions. 
\begin{enumerate}
\item $-1-c_1 \leq \beta \leq 1$.
\item If $c_{1}=-1$, then $c_{2} \leq 2$. If $c_{1} = 0$, then $c_{2} \leq 3$. 
\end{enumerate}
\end{prop}

\begin{proof}
We use the same notation as in Lemma~\ref{lem-comp}. 
Note that 
\begin{align}\label{eq-c2-tmp}
c_{2} \leq 4 \text{ if } c_{1}=-1 \text{ and } c_{2} \leq 3 \text{ if } c_{1}=0
\end{align}
follow from $0 < (-K_{\P_{X}(\mcE)})^{4}=\begin{cases}
       64(-c_2+5) & \mathrm{if}\ c_1=-1\\
       64(-c_2+4) & \mathrm{if}\ c_1=0
\end{cases}
$. 
Here we use Lemma~\ref{lem-comp}~(\ref{lem-comp-IN}) for computing $(-K_{\P_{X}(\mcE)})^{4}$. 

(1) Let us show $-1-c_{1} \leq \beta \leq 1$. 
Since $\mcE^{\vee}\simeq \mcE(-c_1)$, 
we have $0=h^{3}(X,\mcE) = h^{0}(X,\mcE^{\vee}(-2)) = h^{0}(X,\mcE(-c_{1}-2))$ by Lemma~\ref{lem-comp}~(\ref{lem-comp-LP}), 
which implies $\beta \geq -1-c_{1}$. 
On the other hand, by (\ref{eq-RR}), we have 
\[
\chi(\mcE(1))=\begin{cases}
		7 - \frac{3}{2}c_{2}
		& \mathrm{if}\ c_1=-1\\
		12-2c_{2}      
		& \mathrm{if}\ c_1=0. 
\end{cases}
\]
Then (\ref{eq-c2-tmp}) implies $\chi(\mcE (1))>0$. 
Hence
$h^0(\mcE(1)) \geq \chi(\mcE(1)) > 0$ by
Lemma~\ref{lem-comp}~(\ref{lem-comp-LP}), which implies $\beta \leq 1$. 

(2) By (\ref{eq-c2-tmp}), it suffices to show $c_{2} \leq 2$ when $c_{1}=-1$. 
Since $h^{0}(\mcE(1))>0$, we have 
$0 \leq (-K_{\P_{X}(\mcE)})^{3}(\xi+H) = 108-28c_{2}$. 
Moreover, we have $\chi(\mcE)=(1/2)(2-c_{2})$ by Lemma~\ref{lem-comp}~(\ref{lem-comp-RR}). 
Hence $c_{2}$ is an even integer, which implies $c_{2} \leq 2$. 
\end{proof}

\begin{proof}[Proof of Theorem~\ref{thm-I1}]
Let $X$ be a del Pezzo $3$-fold of degree $4$. 
Let $\mcE$ be a normalized rank $2$ weak Fano bundle on $X$. 

(1) 
It is easy to see that if $\mcE$ is one of (i-iv) in Theorem~\ref{mthm}, then $h^{0}(\mcE)>0$, which implies $\beta \leq 0$. 
Hence it suffices to show the converse direction. 

Suppose that $\beta \leq 0$. 
If $\beta < 0$, then by Proposition~\ref{bounds}~(1), we have $c_{1}=0$ and $\beta=-1$. 
Since $-K_{\P_{X}(\mcE)}=2\xi+2H$ is nef, $\mcE(1)$ is nef. 
Thus $\mcI_{Z}$ in the exact sequence (\ref{ex-beta}) is a quotient sheaf of a nef bundle $\mcE(1)$, which implies $Z=\emp$. Hence $\mcE=\mcO_{X}(1) \oplus \mcO_{X}(-1)$. 

Hence we may assume $\beta=0$. 
Then we have $[Z] \sim c_{2}(\mcE)$ for the closed subscheme $Z$ which appears in (\ref{ex-beta}). 
Hence, if $c_{2}=0$, then $Z = \emp$, which implies $\mcE \simeq \mcO \oplus \mcO(c_{1})$, i.e., $\mcE$ belongs to (i) or (ii). 

Hence we may assume that $c_{2}>0$. 
Since $\beta = 0$, $\mcE$ has a global section, which implies that $\xi$ is linearly equivalent to an effective divisor. 
Then by Lemma~\ref{lem-comp}~(\ref{lem-comp-IN}), we have 
\[0 \leq \xi.\left(-K_{\P_{X}(\mcE)} \right)^{3}=\begin{cases}
       -20c_2+4 & \mathrm{if}\ c_1=-1\\
       8(-3c_2+4) & \mathrm{if}\ c_1=0
\end{cases}
. \]
Since $c_{2}>0$, we have $c_{1}=0$ and $c_{2}=1$. Hence $Z$ is a line on $X$. 
Then the exact sequence (\ref{ex-beta}) gives the description (iv) in Theorem~\ref{mthm} (cf. \cite[Lemma~3.2]{ishi}). 
Hence $\mcE$ is one of (i-iv) in Theorem~\ref{mthm}. 
We complete the proof of (1). 
Here we note that the extension $0 \to \mcO_{X} \to \mcE \to \mcI_{L} \to 0$ for an arbitrary line $L$ on $X$ is uniquely determined by $L$ since $\Ext^{1}(\mcI_{L},\mcO_{X}) \simeq H^{2}(\mcI_{L} \otimes \omega_{X})^{\vee} \simeq H^{1}(\omega_{X}|_{L})^{\vee} \simeq \C$ by the Serre duality and the exact sequence $0 \to \mcI_{L} \otimes \omega_{X} \to \omega_{X} \to \omega_{X}|_{L} \to 0$.

(2) First we show that $(c_{1},c_{2}) \in \{(-1,2),(0,2),(0,3)\}$ assuming $\beta>0$. 
In this case, we have $\chi(\mcE) \leq 0$ by Lemma~\ref{lem-comp}~(\ref{lem-comp-LP}). 
If $c_{1}=-1$, then $\chi(\mcE)=(1/2)(2-c_{2})$ by  Lemma~\ref{lem-comp}~(\ref{lem-comp-RR}) and hence $c_{2}=2$ by Proposition~\ref{bounds}. 
If $c_{1}=0$, then $\chi(\mcE)=2-c_{2}$ and hence $c_{2} \in \{2,3\}$ by Proposition~\ref{bounds}. 
The inverse direction holds by (1). 
We complete the proof. 
\end{proof}

\section{Characterization of slope stable weak Fano bundles with $c_{1}=-1$: Proof of Theorem~\ref{thm-H}}\label{sec-H}
We devote this section to the proof of Theorem~\ref{thm-H}. 

\begin{proof}[Proof of Theorem~\ref{thm-H}]
Let $X$ be a del Pezzo $3$-fold of degree $4$. 
We will show the implications (2) $\ra$ (1), (1) $\ra$ (3), and (3) $\ra$ (2) in Theorem~\ref{thm-H}. 

The implication (2) $\ra$ (1) is obvious. 

We show the implication (1) $\ra$ (3). 
Let $\mcE$ be a weak Fano bundle with $(c_1, c_2) = (-1, 2)$. 
Then by Theorem~\ref{thm-I1} and Proposition~\ref{bounds}, we have 
$h^{0}(\mcE(1))>0$ and $h^{0}(\mcE)=0$. 
Let $s \in H^0(X, \mcE(1))$ be a non-zero section and $C = Z(s)$ the zero locus of the section $s$. 
Then we have an exact sequence
\begin{align}\label{ex-Hara}
 0 \to \mcO_X \xrightarrow{s} \mcE(1) \to \mcI_{C/X}(1) \to 0. 
\end{align}
By the Riemann-Roch theorem and the additivity of the Euler character,
we have that the Hilbert polynomial of $C$ is $2t+1$, where $t$ is a variable. 
Hence $C$ is a conic curve that may be singular. 

By the above argument, 
in order to show that $\mcE$ is (v) in Theorem~\ref{mthm}, 
it is enough to show that $\mathcal{E}(1)$ is globally generated.
Let $P$ be the linear span of $C$ in $\mathbb{P}^5$.
Since $C$ is a conic, $P$ is a plane (i.e. two-dimensional linear subspace). 
Choose two smooth quadrics $Q_1$ and $Q_2$ such that $X = Q_1 \cap Q_2$.
Then at least one of $Q_1$ and $Q_2$ does not contain $P$.
Indeed, if both of $Q_1$ and $Q_2$ contain $P$, then $P$ is a divisor of $X$, 
which contradicts that $\Pic(X) \simeq \Z$ generated by a hyperplane section $H_{X}$. 
Therefore we may assume that the plane $P$ is not a subscheme of $Q_1$.
In this case, the intersection $Q_1 \cap P$ is a conic in $P$.
Since
$C \subset X \cap P \subset Q_{1} \cap P$, we notice that $C = X \cap P$.
This means that $C$ is defined by linear equations in $X$, and hence $\I_{C/X}(1)$ is globally generated.
Then the exact sequence (\ref{ex-Hara}) implies that $\mcE(1)$ is globally generated. 
We complete the proof of this implication (1) $\ra$ (3). 

Finally, we show the implication (3) $\ra$ (2). 
Let $\mcE$ be (v) in Theorem~\ref{mthm}, that is, 
$\mcE$ is the non-trivial extension of $\mcI_{C}$ by $\mcO_{X}(-1)$, 
where $C$ is a smooth conic. 
As we saw in the proof of the implication (1) $\ra$ (3), 
$\mcI_{C}(1)$ is globally generated and hence so is $\mcE(1)$. 
Since $\mathcal{E}(1)$ is globally generated, it is nef and hence $\xi + H$ is a nef divisor on $\mathbb{P}_X(\mathcal{E})$. 
Thus $-K_{\P_{X}(\mcE)}=2\xi+3H$ is ample, which means that $\mcE$ is a Fano bundle. 

Here we note that the extension (\ref{ex-Hara}) is uniquely determined for a smooth conic $C$ since $\Ext^{1}(\mcI_{C/X}(1),\mcO_{X}) \simeq H^{2}(\mcI_{C/X}(-1))^{\vee} \simeq H^{1}(\mcO_{X}(-1)|_{C})^{\vee} \simeq \C$ by the Serre duality and the exact sequence $0 \to \mcI_{C/X}(-1) \to \mcO_{X}(-1) \to \mcO_{X}(-1)|_{C} \simeq \mcO_{\P^{1}}(-2) \to 0$. 
We complete the proof of Theorem~\ref{thm-H}.
\end{proof}


\section{Global generation of nef bundles with $c_{1}=-K_{X}$}

The main purpose of this section is to prove Theorem~\ref{F-thm-glgen} and Theorem~\ref{thm-F-chara}. 

\subsection{Ladders of weak Mukai manifolds}
For proving Theorem~\ref{F-thm-glgen}, we will use the existence of a smooth ladder on a weak Mukai manifold, which is nothing but the following theorem. 

\begin{thm}\label{thm-Mella}\cite{Mella,Shin}
Let $n \geq 3$ and $M$ an $n$-dimensional weak Mukai manifold, i.e., $M$ is an $n$-dimensional weak Fano manifold with a nef big divisor $A$ such that $-K_{M} \sim (n-2)A$. 
Then $|A|$ has a smooth member. 
Moreover, if $\Bs |A|$ is not empty, then $\Bs|A|$ is a smooth rational curve. 
\end{thm}

\begin{proof}
If $|A|$ is base point free, then Theorem~\ref{thm-Mella} follows from Bertini's theorem. 
Hence we may assume that $\Bs|A| \neq \emp$. 
When $n=3$, this theorem is proved by Minagawa \cite[Theorem~3.1]{Minagawa}. 
Hence we may assume that $n \geq 4$. 

Let $\psi \colon M \to \ol{M}$ be the crepant birational contraction onto the anti-canonical model $\ol{M}$ of $M$, which is given by applying the Kawamata-Shokurov base point free theorem to $A$. 
Let $\ol{A}$ be an ample Cartier divisor on $\ol{M}$ such that $A=\psi^{\ast}\ol{A}$. 
Then by \cite[Theorem~2.3]{Mella}, a general member $\ol{D} \in |\ol{A}|$ has only canonical singularities. 
By the inversion of adjunction \cite[Theorem~5.50]{KM98}, the log pair $(\ol{M},\ol{D})$ is plt. 
Then $D:=\psi^{\ast}\ol{D}$ is irreducible and reduced since $\psi$ is birational. 
Since $K_{M}+D=\psi^{\ast}(K_{\ol{M}}+\ol{D})$,  $(M,D)$ is also plt. 
Hence by \cite[Corollary~5.52 and Theorem~5.50]{KM98}, $D$ is normal and has only log terminal singularities. 
Since $D$ is Gorenstein, $D$ has only canonical singularities. 
This argument holds even if $M$ has Gorenstein canonical singularities. 

Thus, by repeating this argument and using Reid's result \cite[Theorem]{Reid}, we have a ladder $M_{n-2} \subset M_{n-3} \subset \cdots \subset M_{1} \subset M_{0} = M$ such that each $M_{i+1} \in |A|_{M_{i}}|$ has only Gorenstein canonical singularities. 
We may assume that $M_{i}$ is smooth on the outside of $\Bs|A|$ by Bertini's theorem. 
We set $S=M_{n-2}$ and $X=M_{n-3}$. 
Note that $X$ is a weak Fano $3$-fold such that $A|_{X}=-K_{X}$ and $\Bs|A|=\Bs|-K_{X}|$. 
By Shin's result \cite[Theorem~0.5]{Shin}, it is known that $\dim \Bs|-K_{X}| \in \{0,1\}$. 

Now we show $\dim \Bs|A| \neq 0$. 
If $\dim \Bs|A|=0$, then it also follows from 
[ibid.,Theorem~0.5] that $\Bs|A|=\Bs|-K_{X}|$ consist of a single point $p \in M$ and $p$ is an ordinary double point of $S$. 
Moreover, $p$ must be a singular point of $X$. 
Then by the same argument as \cite[Proof~of~Theorem~3.1]{Minagawa}, 
we can show that $\psi|_{S} \colon S \to \psi(S)$ is isomorphic at $p$. 
Since $S=\psi^{-1}(\psi(S))$, $\psi$ is isomorphic at $p$. 
In particular, $\ol{M}$ is smooth at $\ol{p}:=\psi(p)$. 
By the same argument as \cite[Proof~of~Theorem~2.5]{Mella}, 
there is a member $\ol{D} \in |\ol{A}|$ which is smooth at $\ol{p}$. 
Since $\psi$ is isomorphic at $p$, $D:=\psi^{-1}(\ol{D}) \in |A|$ is smooth at $p=\Bs|A|$. 
Hence general members of $|A|$ are smooth, 
which contradicts that $X=M_{n-3}$ is singular at $p$. 

Therefore, $\dim \Bs|-K_{X}|=\dim \Bs|A|=1$. 
Again by \cite[Theorem~0.5]{Shin}, $\Bs|-K_{X}|$ is a smooth rational curve and $|-K_{X}|$ has a member which is smooth along $\Bs|-K_{X}|$. 
Hence a general member of $|A|$ is smooth. 
We complete the proof. 
\end{proof}

\subsection{Lower bound of the degree of elliptic curves on Fano 3-folds}

As another preliminary for proving Theorem~\ref{F-thm-glgen}, we give a (non-optimal) lower bound of the anti-canonical degree of an elliptic curve on a Fano $3$-fold as follows. 
\begin{lem}\label{lem-degell1}
Let $X$ be a Fano $3$-fold of Picard rank $1$ and $B \subset X$ be an elliptic curve. 
Let $r(X):=\max \{r \in \Z_{>0} \mid -K_{X} \sim rH \text{ for some divisor } H \}$ be the Fano index of $X$. 
Then we obtain the following assertions. 
\begin{enumerate}
\item If $r(X)=1$, then $(-K_{X}).B \geq \frac{(1/2)(-K_{X})^{3}+3}{2}$. 
\item If $r(X) \leq 2$, then $(-K_{X}).B \geq \frac{1}{4}(-K_{X})^{3}$. 
\item If $r(X) \geq 3$, then $(-K_{X}).B \geq 12$. 
\end{enumerate}
\end{lem}
\begin{proof}
(1) and (2): First we treat the case $r(X) \leq 2$. 
By the Hartshorne-Serre correspondence, 
there is a rank $2$ vector bundle $\mcG$ fitting into
$0 \to \mcO_{X} \to \mcG \to \mcI_{B}(-K_{X}) \to 0$. 
This exact sequence implies $h^{0}(\mcG(K_{X}))=0$, which implies that $\mcG$ is slope stable (resp. slope semi-stable) if $r(X)=1$ (resp. $2$). 
Then (1) follows from \cite[Lemma~3.1]{BF08} and 
(2) follows from the Bogomolov inequality $(c_{1}(\mcG)^{2}-4c_{2}(\mcG))(-K_{X}) \leq 0$. 

(3): When $r(X) \geq 3$, it is known by Kobayashi and Ochiai that $X$ is isomorphic to a hyperquadric $\Q^{3}$ in $\P^{4}$ or the projective $3$-space $\P^{3}$.  
It is easy to see that the degree of every elliptic curve $B$ on $\Q^{3}$ (resp. $\P^{3}$) is greater than or equal to $4$ (resp. $3$). 
Thus we obtain (3). 
\end{proof}

\subsection{Proof of Theorem~\ref{F-thm-glgen}}\label{subsec-proofofglgen}
Let $X$ be a smooth Fano $3$-fold of Picard rank $1$ and $\mcF$ a nef bundle with $\det \mcF=\mcO(-K_{X})$ and $\rk \mcF \geq 2$. 
Let $\pi \colon M:=\P_{X}(\mcF) \to X$ be the natural projection and $\xi$ a tautological divisor. 
Set $n=\dim M = 2+\rk \mcF$. 
Recall the assumption $\xi^{\dim M}=c_{1}(\mcF)^{3}-2c_{1}(\mcF)c_{2}(\mcF)+c_{3}(\mcF)>0$. 
Since $-K_{M} \sim (\rk \mcF) \xi$, $M$ is a weak Mukai manifold. 
Then Theorem~\ref{thm-Mella}
gives the following smooth ladder of $|\xi|$: 
\begin{align}\label{eq-ladder}
M_{n-2} \subset M_{n-3} \subset \cdots \subset M_{1} \subset M_{0}=M=\P_{X}(\mcF),
\end{align}
where $M_{i}$ is a smooth prime member of $|\xi|_{M_{i-1}}|$. 

Until the end of this proof, we suppose that $\mcF$ is not globally generated. 
We proceed in $5$ steps. 

\subsubsection{Step 1.}
In Step 1 and Step 2, we prepare some facts on the weak Fano $3$-fold $M_{n-3}$ and the K3 surface $M_{n-2}$. 

\begin{claim}\label{claim-weakFano3}
Set $\wt{X}:=M_{n-3}$, which is a weak Fano $3$-fold. 
Then $\pi_{\wt{X}}:=\pi|_{\wt{X}} \colon \wt{X} \to X$ is the blowing-up along
a (possibly disconnected) smooth curve $C$ on $X$. 
Moreover, the following assertions hold. 
\begin{enumerate}
\item $c_{2}(\mcF) \equiv C$ and $c_{3}(\mcF)=-2\chi(\mcO_{C})$. 
\item Every connected component of $C$ is of genus $g \geq 1$. 
\end{enumerate}
\end{claim}
\begin{proof}
Let $s_{1},\ldots,s_{\rk \mcF-1} \in H^{0}(X,\mcF) \simeq H^{0}(M,\mcO(\xi))$ be the sections defining $\wt{X}$. 
Set $s=(s_{1},\ldots,s_{\rk \mcF-1}) \colon \mcO^{\oplus \rk F-1} \to \mcF$ and let $C:=\{x \in X \mid \rk s(x)<\rk \mcF-1\}$ be the degeneracy locus of $s$, which is possibly empty. 
Since $\wt{X}$ is a smooth variety, $\wt{X}$ is the blowing-up of $X$ along $C$ with $\dim C \leq 1$ (see \cite[Lemma~6.9]{Fukuoka18}). 
Moreover, we obtain the following exact sequence: 
\begin{align}\label{F-ex-ext}
0 \to \mcO^{\oplus \rk \mcF-1} \mathop{\to}^{s} \mcF \to \mcI_{C}(-K_{X}) \to 0.
\end{align}
By the exact sequence (\ref{F-ex-ext}), we have 
$\mathcal{E}xt^{i}(\mcO_{C},\mcO_{X})=\mathcal{E}xt^{i-1}(\mcI_{C},\mcO_{X})=0$ for $i \geq 3$. 
Hence $C$ is purely 1-dimensional or empty. 
If $C$ is empty, then we have $\mcF=\mcO^{\oplus \rk\mcF-1} \oplus \mcO(-K_{X})$, which is globally generated since $X$ is of Picard rank $1$ (cf. \cite[Corollary~2.4.6]{Fanobook}). 
Hence we may assume that $C$ is not empty. 
Moreover, it follows from \cite[Theorem~(3.3)]{Mori82} that $C$ is smooth. 

(1) The equality $C \equiv c_{2}(\mcF)$ follows from the exact sequence (\ref{F-ex-ext}). 
Moreover, since $c_{1}(\mcF)^{3}-2c_{1}(\mcF)c_{2}(\mcF)+c_{3}(\mcF)=\xi^{\rk\mcF+2}=(-K_{\wt{X}})^{3}=(-K_{X})^{3}+2K_{X}.C-2\chi(\mcO_{C})$, 
we obtain that $c_{3}(\mcF)=-2\chi(\mcO_{C})$. 

(2) We assume that a connected component $C_{i}$ of $C$ is a smooth rational curve. 
Taking the restriction of the exact sequence (\ref{F-ex-ext}) to $C_{i}$, we obtain an exact sequence $\displaystyle \mcO_{\P^{1}}^{\oplus \rk \mcF-1} \mathop{\to}^{s|_{C_{i}}} \mcF|_{C_{i}} \to \mcN_{C_{i}/X}^{\vee}(-K_{X}|_{C_{i}}) \to 0$. 
Since $\det \mcF=\mcO(-K_{X})$ and 
$\deg( \mcN_{C_{i}/X}^{\vee}(-K_{X}|_{C_{i}}) )
=\deg(-K_{C_{i}}) - \deg (K_{X}|_{C_{i}})$ by the adjunction formula, 
the degree of $\Im(s|_{C_{i}})$ is $-2$, which contradicts that $\Im(s|_{C_{i}})$ is a quotient of $\mcO_{\P^{1}}^{\oplus \rk \mcF-1}$. 
We complete the proof. 
\end{proof}

\subsubsection{Step 2.}

\begin{claim}\label{claim-K3}
Set $S:=M_{n-2}$, which is a K3 surface. 
Set $\ol{S}=\pi(S)$ and $\pi_{S}:=\pi|_{S} \colon S \to \ol{S}$. 
Then the following assertions hold. 

\begin{enumerate}
\item There is an elliptic fibration $f_{S} \colon S \to \P^{1}$ and a section $\Gamma$ of $f_{S}$ such that $\xi|_{S} \sim gB+\Gamma$ for a general fiber $B$ of $f_{S}$ . Moreover, it holds that 
\begin{align}\label{eq-BsLocus}
\Bs|\xi|=\Bs|\xi|_{S}|=\Gamma. 
\end{align}
\item It holds that $g = \frac{1}{2}(-K_{X})^{3}+K_{X}.C-\chi(\mcO_{C})+1 \in \Z_{\geq 2}$. 
\item 
$\ol{S}$ is normal and contains the curve $C$ defined in Claim~\ref{claim-weakFano3}. 
\item If $\Gamma$ is not contracted by $\pi_{S}$, then $\pi_{S}$ is isomorphic along a general $f_{S}$-fiber $B$. 
In particular, a smooth elliptic curve $\ol{B}:=\pi_{S}(B)$ is an effective Cartier divisor on $\ol{S}$. 
Moreover, it holds that 
\begin{align}\label{eq-importantK3}
\ol{B}.C=(-K_{X}|_{\ol{S}}).\ol{B}-1.
\end{align}
\end{enumerate}
\end{claim}

\begin{proof}
Since we assume that $\Bs|\xi|$ is not empty, so is $\Bs|\xi|_{S}|$, where $\xi|_{S}$ is a nef and big divisor on $S$. 
Then (1) is known by \cite[Lemma~2.3]{Shokurov} and \cite[(2.7.3) and (2.7.4)]{SD74}. 
(2) immediately follows from $c_{1}(\mcF)^{3}-2c_{1}(\mcF)c_{2}(\mcF)+c_{3}(\mcF)=\xi^{\rk\mcF +2} = (\xi|_{S})^{2}=(gB+\Gamma)^{2} = 2g-2 > 0$ and Claim~\ref{claim-weakFano3}~(1). 

(3) 
Set $E_{S}:=\Exc(\pi_{\wt{X}}) \cap S$ where $\pi_{\wt{X}}$ is the blowing-up defined in Claim~\ref{claim-weakFano3}. 
Then $E_{S}$ is a member of some tautological bundle of the $\P^{1}$-bundle $\Exc(\pi_{\wt{X}}) \to C$. 
Hence $\ol{S}$ contains $C$ and every fiber of $S \to \ol{S}$ is connected, 
which implies $\ol{S}$ is normal. 


(4) Suppose $\Gamma$ is not contracted by $\pi_{S}$. 
Then for every exceptional curve $l$ of $\pi_{S}$, we have $1=\xi.l=(gB+\Gamma).l \geq g(B.l)$. 
Since $g \geq 2$, we have $B.l=0$. 
Hence $\pi_{S}$ is isomorphic along $B$. 
Thus $\ol{B}$ is a Cartier divisor on $\ol{S}$ and it holds that $B=\pi_{S}^{\ast}\ol{B}$. 
Since $-K_{\wt{X}}|_{S}=\xi|_{S}=gB+\Gamma$, we have 
$\ol{B}.C=\pi_{S}^{\ast}\ol{B}.E_{S} = B.(\pi_{S}^{\ast}(-K_{X}|_{S})-(gB+\Gamma)) = (-K_{X}).\ol{B}-1$, 
which implies the equality (\ref{eq-importantK3}). 
\end{proof}

\subsubsection{Step 3.}
In this step, we show the following claim. 

\begin{claim}\label{claim-degell2}
If $\Gamma$ is not contracted by $\pi_{S}$, then, for a general $f_{S}$-fiber $B$, we have 
\begin{align}\label{ineq-Bogomolov1}
\frac{1}{4}(-K_{X})^{3} \leq (-K_{X}).\ol{B},
\end{align}
where $\ol{B}=\pi_{S}(B)$.  
\end{claim}
\begin{proof}
Since $\Gamma$ is not contracted by $\pi_{S}$, 
$\ol{B}:=\pi_{S}(B)$ is a smooth elliptic curve by Claim~\ref{claim-K3}~(4). 
The inequality $(-K_{X}).\ol{B} \geq (1/4)(-K_{X})^{3}$ follows from Lemma~\ref{lem-degell1} when $r(X) \leq 2$. 
Assume that $r(X) \geq 3$, that is, $X=\P^{3}$ or $\Q^{3}$. 
Let $H_{X}$ be an ample generator of $\Pic(X)$ and $r:=r(X)$. 
Set $H=(\pi_{X}^{\ast}H_{X})|_{S}$ and $A:=rH \sim (\pi_{X}^{\ast}(-K_{X}))|_{S}$. 
To obtain a contradiction, we assume $(-K_{X}).\ol{B} < (1/4)(-K_{X})^{3}$. 
Then Lemma~\ref{lem-degell1} gives $(-K_{X}).\ol{B}=12$. 
When $X=\P^{3}$ (resp. $\Q^{3}$), $\ol{B}$ is a plane cubic curve (resp. a complete intersection of two quadrics in $\P^{3}$) and hence $H-B$ is effective. 
Let $\wt{C} \subset S$ be the proper transform of $C \subset \ol{S}$. 
Then $\wt{C}$ is a nef divisor on $S$ by Claim~\ref{claim-weakFano3}~(2). 
Hence we have $H.\wt{C}=(H-B+B).\wt{C} \geq B.\wt{C}=A.B-1$ by the equality (\ref{eq-importantK3}). 
Since $\Gamma$ is not contracted by $\pi_{S}$, 
we have $H.\Gamma \geq 1$ and hence $gB.H+1 \leq (gB+\Gamma).H=(A-\wt{C}).H \leq A.H-A.B+1$. 
Hence we have $(A+gH).B \leq A.H$. 
Since $A=rH$, we have $H.B \leq \frac{r}{r+g}H^{2}$.  
When $X=\P^{3}$ (resp. $\Q^{3}$), we obtain $H.B \leq \frac{16}{4+g} < 3$ (resp. $H.B \leq \frac{12}{3+g} < 3$) since $g \geq 2$, which is a contradiction. 
Hence we have $(-K_{X}).\ol{B} \geq (1/4)(-K_{X})^{3}$. 
\end{proof}

\subsubsection{Step 4.}\label{subsubsec-GammaContracted}
In this step, we show the following key claim. 
\begin{claim}\label{claim-GammaContracted}
The base locus $\Bs|\xi|=\Gamma$ is contracted by $\pi$. 
\end{claim}

\begin{proof}[Proof of Claim~\ref{claim-GammaContracted}]
We derive a contradiction by assuming $\pi_{X}^{\ast}(-K_{X}).\Gamma \geq 1$. 
By Claim~\ref{claim-K3}~(1), we have $\xi|_{S} \sim gB+\G$ on $S$ 
and hence 
\begin{align}\label{ineq-Bogomolov2}
g(-K_{X}).\ol{B}+1 
&\leq \pi_{S}^{\ast}(-K_{X}|_{\ol{S}}).(gB+\Gamma)
=\pi^{\ast}(-K_{X}).\xi^{\dim \P_{X}(\mcF)-1}  \\
&=c_{1}(\mcF)^{3}-c_{1}(\mcF)c_{2}(\mcF) 
=(-K_{X})^{3}-(-K_{X}).C \notag
\end{align}
by Claim~\ref{claim-weakFano3}~(1). 
Then Claim~\ref{claim-K3}~(2) implies 
\begin{align}\label{ineq-t2andc3}
(-K_{X})^{3}-(-K_{X}).C
&=(g-1)+\frac{(-K_{X})^{3}}{2} + \chi(\mcO_{C}) \leq (g-1)+\frac{(-K_{X})^{3}}{2}. 
\end{align}
Then 
the inequalities (\ref{ineq-Bogomolov1}), (\ref{ineq-Bogomolov2}), and (\ref{ineq-t2andc3}) imply $g\frac{(-K_{X})^{3}}{4}+1 \leq (g-1)+\frac{(-K_{X})^{3}}{2}$, 
which is equivalent to 
\begin{align}\label{ineq-sandwich}
(g-2)((-K_{X})^{3}-4) \leq 0. 
\end{align}
Noting that $-K_{X}.\ol{B} \geq 2$ since $\mcO_{X}(-K_{X})$ is globally generated (cf. \cite[Corollary~2.4.6]{Fanobook}). 
Hence we also obtain $2g+1 \leq (g-1)+\frac{(-K_{X})^{3}}{2}$ from (\ref{ineq-Bogomolov2}) and (\ref{ineq-t2andc3}), which implies 
$8 \leq 2(g+2) \leq (-K_{X})^{3}$. 
Then (\ref{ineq-sandwich}) implies $g=2$. 
Hence the inequalities (\ref{ineq-Bogomolov1}), (\ref{ineq-Bogomolov2}), and (\ref{ineq-t2andc3}) are equalities. 
Thus we have the following assertions: 
\begin{itemize}
\item[(a)] $-K_{X}.\ol{B}=(-K_{X})^{3}/4 \geq 2$ and hence $(-K_{X})^{3}$ is divisible by $4$. 
\item[(b)] $\chi(\mcO_{C})=0$ and hence $2(-K_{X}.C)=(-K_{X})^{3}-2$ by Claim~\ref{claim-K3}~(2). 
\end{itemize}
If follows from (a) and Lemma~\ref{lem-degell1}~(1) that $r(X) \geq 2$. 
It also follows from (a) and (b) that $-K_{X}.C$ is an odd number. 
In particular, the Fano index $r(X)$ of $X$ is $3$, i.e., $X$ is a hyperquadric, which contradicts (a). 
Hence we obtain a contradiction and conclude that $\pi^{\ast}(-K_{X}).\Gamma=0$, which means that $\Gamma$ is contracted by $\pi$. 
\end{proof}

\subsubsection{Step 5.}
Finally, we prove Theorem~\ref{F-thm-glgen}. 
Let $\tau \colon \wt{M}:=\Bl_{\G}M \to M=\P_{X}(\mcF)$ be the blowing-up. 
Set $D:=\Exc(\tau)$ and $\wt{\xi}:=\tau^{\ast}\xi-D$. 
Since $\xi|_{S} \sim gB+\G$, $|\wt{\xi}|$ is base point free. 
Let $\Psi \colon \wt{M} \to \P^{g+\rk \mcF}$ be the morphism given by $|\wt{\xi}|$ and $W$ its image. 
Then, for the proper transform $\wt{S} \subset \wt{M}$ of $S$, the morphism $\wt{S} \to \Psi(\wt{S})$ is given by $|gB|$. 
Since $\Psi(\wt{S})$ is a linear section of $W$, $W$ is of degree $g$. 
In particular, the $\Delta$-genus of $W \subset \P^{g+\rk \mcF}$ is $0$. 
Moreover, $D \to W$ is birational. 

Let us compute $g=\deg W$ and the normal bundle $\mcN_{\G/M}$. 
By Claim~\ref{claim-GammaContracted}, $\Gamma \subset \wt{X}$ is contracted by $\pi_{X}$. 
Then Claim~\ref{claim-weakFano3} implies that $\Gamma$ is a line in a fiber of $M=\P_{X}(\mcF) \to X$. 
Therefore, it holds that $\mcN_{\G/M} \simeq \mcO_{\P^{1}}(1)^{\oplus \rk \mcF-2} \oplus \mcO_{\P^{1}}^{\oplus 3}$. 
Hence $D$ is isomorphic to $\P_{\P^{1}}(\mcO^{\oplus \rk \mcF-2} \oplus \mcO(1)^{\oplus 3})$. 
Under this identification, $\mcO(\wt{\xi})|_{D}$ is isomorphic to the tautological bundle. 
Since $\Psi|_{D} \colon D \to W$ is birational, we have $g=\deg W=\wt{\xi}^{\rk \mcF+1}.D=3$. 
Since the $\Delta$-genus of $W$ is $0$, 
$D=\P_{\P^{1}}(\mcO(1)^{\oplus 3} \oplus \mcO^{\oplus \rk \mcF-2}) \to W$ is the morphism given by the complete linear system of the tautological bundle and $W$ is a join of $\P^{\rk \mcF-3}$ and $\P^{1} \times \P^{2}$ in $\P^{\rk \mcF+3}$. 

Then we let $V=M_{n-4} \subset M=\P_{X}(\mcF)$ be the intersection of general $(r-2)$ members of $|\xi|$.
$V$ contains $\G$ since $\G$ is the base locus of $|\xi|$. 
Let $\tau_{V} \colon \wt{V}:=\Bl_{\G}V \to V$ be the blowing-up and $D_{V}:=\Exc(\wt{V} \to V)$. 
Then $|\tau_{V}^{\ast}\xi-D_{V}|$ induces the morphism $\psi \colon \wt{V} \to \P^{1} \times \P^{2}$. 
Note that $D_{V} \to \P^{1} \times \P^{2}$ is isomorphic. 
Let $\wt{F_{1}},\wt{F_{2}}$ be two general fibers of $\pr_{1} \circ \psi \colon \wt{V} \to \P^{1}$. 
For each $i \in \{1,2\}$, we set $F_{i}:={\tau_{V}}_{\ast}\wt{F_{i}}$. 
Since $D_{V} \simeq \P^{1} \times \P^{2}$ and $D_{V} \to \G$ coincides with the first projection, 
it holds that $F_{1} \cap F_{2}=\emp$. 
Hence $|F_{i}|$ is base point free and induces a morphism $h \colon V \to \P^{1}$. 
Thus we obtain the following commutative diagram: 

\[\xymatrix{
&D_{V}\ar[ld] \ar@{^{(}->}[r] \ar@(u,u)[rr]^{\simeq}&\wt{V} \ar[ld]_{\tau_{V}} \ar[r] & \P^{1} \times \P^{2} \ar[ld]^{\pr_{1}} \\
\Gamma  \ar@{^{(}->}[r] &V \ar[r]_{h} \ar[d]_{\pi|_{V}} & \P^{1} \\
&X
}\]
Recall that the $h$-section $\Gamma$ is contracted by $\pi|_{V} \colon V \to X$, which is an adjunction theoretic scroll. 
We set $x:=\pi|_{V}(\G)$ and $J:=\pi|_{V}^{-1}(x)$. 
Then $J$ is a projective space containing $\G$. 
If $\dim J \geq 2$, then $h|_{J} \colon J \to \P^{1}$ must contract the whole $J$, which contradicts that $\G$ is a section of $h$. 
Hence $J=\G$.  
Therefore, the induced morphism $(\pi|_{V},h) \colon V \to X \times \P^{1}$ is birational. 
Since every fiber of $\pi|_{V} \colon V \to X$ is a projective space, $V$ is isomorphic to $X \times \P^{1}$. 

Let $\alpha \colon \mcO_{X}^{\oplus r-2} \to \mcF$ be the morphism corresponding to $(r-2)$ members of $|\xi|$ which define $V \subset \P_{X}(\mcF)$. 
Since $V \simeq X \times \P^{1}$, $\alpha$ is injective and $\Cok \alpha \simeq \mcL^{\oplus 2}$ for some line bundle $\mcL$. 
Since $\det \mcF=\mcO(-K_{X})$, $-K_{X}$ is divided by $2$ in $\Pic(X)$ and $\mcL=\mcO_{X}(\frac{-K_{X}}{2})$. 
Hence it holds that $\mcF \simeq \mcO_{X}^{\oplus r-2} \oplus \mcO_{X}(\frac{-K_{X}}{2})^{\oplus 2}$. 
Since $\mcF$ is not globally generated, so is $\mcO_{X}(\frac{-K_{X}}{2})$.
By the classification of del Pezzo $3$-folds, $X$ is a del Pezzo $3$-fold of degree $1$ (cf. \cite{Fujita80}). 
We complete the proof of Theorem~\ref{F-thm-glgen}. 
\qed
\subsection{Proof of Theorem~\ref{thm-F-chara}}

In this section, we give a corollary as consequences of the above results 
and show Theorem~\ref{thm-F-chara}.
\begin{cor}\label{cor-instanton}
Let $X$ be a del Pezzo $3$-fold of Picard rank $1$. 
Let $\mcE$ be a rank $2$ weak Fano bundle on $X$ with $c_{1}(\mcE)=0$. 
Then the zero scheme of a general global section of $\mcE(1)$ is a connected smooth elliptic curve and $h^{1}(\mcE(-1))=0$. 
Especially, if $\mcE$ is slope stable, then $\mcE$ is an instanton bundle \cite{Kuznetsov2012,Faenzi2014}. 
\end{cor}
\begin{proof}
First, we treat the case $X$ is of degree $1$ and $\mcE=\mcO_{X}^{\oplus 2}$. 
In this case, it is well-known that $X$ is a smooth sextic hypersurface of $\P(1,1,1,2,3)$ \cite{Fujita80} and hence the zero scheme of a general section $s \in H^{0}(\mcE(1))=H^{0}(\mcO_{X}(1))^{\oplus 2}$, which is a general complete intersection of two members of $|\mcO_{X}(1)|$, is a connected smooth elliptic curve.
Hence, by Theorem~\ref{F-thm-glgen}, we may assume that $\mcF := \mcE(1)$ is globally generated. 
Let $C$ be the zero scheme of a general $s \in H^{0}(\mcF)$. 
Since $C$ is a smooth curve with $\omega_{C} \simeq \mcO_{C}$, $C$ is a disjoint union of smooth elliptic curves.
We show that $C$ is connected. 
Let $d$ be the degree of $X$ and $\xi_{\mcF}$ be the tautological bundle of $\P_{X}(\mcF)$. 
Since $\xi_{\mcF}^{4} = c_{1}(\mcF)^{3}-2c_{1}(\mcF)c_{2}(\mcF)=8d-4\deg C>0$, 
we have $\deg C < 2d$. 
By Lemma~\ref{lem-degell1}~(2), 
the degree of every connected component of $C$ is greater than or equal to $d$. 
Therefore, $C$ must be connected. 
Using an exact sequence $0 \to \mcO \to \mcF \to \mcI_{C}(2) \to 0$, 
we have $H^{1}(\mcE(-1))=H^{1}(\mcF(-2))=H^{1}(\mcI_{C})=0$ since $C$ is a smooth connected curve and $H^{1}(\mcO_{X})=0$. 
Therefore, if $\mcE$ is slope stable, then $\mcE$ is an instanton bundle by definition. 
\end{proof}

\begin{proof}[Proof of Theorem~\ref{thm-F-chara}]
The implication (1) $\ra$ (2) follows from Corollary~\ref{cor-instanton} and Theorem~\ref{F-thm-glgen}. 
The implication (2) $\ra$ (1) for (vi) (resp. (vii)) holds since it follows from \cite{Hulek} (resp. our assumption in (vii)) that the curve $C$ can be defined by quadratic equations. 
We complete the proof. 
\end{proof}

\section{Existence: Proof of Theorem~\ref{thm-existence}}
In this section, we show the existence of vector bundles for each case in Theorem~\ref{mthm} on an arbitrary del Pezzo $3$-fold of degree $4$. 
The existence of the cases (i-iii) is clear, and that of the case (iv) and (v) was proved in Proof of Theorem~\ref{thm-I1} in Section~\ref{sec-I1} and Proof of Theorem~\ref{thm-H} in Section~\ref{sec-H} respectively. 
The existence of the case (vi) is equivalent to the existence of a special Ulrich bundle of rank $2$ (cf. Remark~\ref{rem-main}~(3)), which is proved by \cite[Proposition 6.1]{Beauville}. 
Thus, the remaining case is only (vii). The main purpose of this section is to prove the existence of a bundle belonging to (vii), which is equivalent to the following theorem by the Hartshorne-Serre correspondence.

\begin{thm}[$=$Theorem~\ref{thm-existence} for (vii)]\label{7-mainthm}
Let $X \subset \P^{5}$ be an arbitrary smooth complete intersection of two hyperquadrics. 
Then $X$ contains an elliptic curve $C$ of degree $7$ defined by quadratic equations. 
\end{thm}

The key proposition is the following. 
\begin{prop}\label{7-stability}
Let $C \subset \P^{5}$ be an elliptic curve of degree $7$. 
Then the following assertions are equivalent. 
\begin{enumerate}
\item $C$ is defined by quadratic equations. 
\item $C$ has no trisecants. 
\item For every rank $3$ subsheaf $\mcE$ of $\Omega_{\P^{5}}(1)|_{C}$, we have $\mu(\mcE)<-7/5$. 
In this article, we call this property the \emph{slope stability for rank $3$ subsheaves}. 
\end{enumerate}
\end{prop}

By Proposition~\ref{7-stability}, it suffices to show that every smooth complete intersection $X$ of two hyperquadrics in $\P^{5}$ contains an elliptic curve $C$ of degree $7$ having no trisecants. 
We will obtain such a curve by smoothing the union of a conic and an elliptic curve of degree $5$. 

\subsection{Mukai's technique}

Our proof of the implication (3) $\ra$ (1) in Proposition~\ref{7-stability} is based on Mukai's technique \cite{Mukai}.
We quickly review his technique. 

Let $C$ be an elliptic curve of degree $7$ in $\P^{5}$. 
Let $\s \colon \wt{\P}:=\Bl_{C}\P^{5} \to \P^{5}$ be the blowing-up. 
Set $H:=\s^{\ast}\mcO_{\P^{5}}(1)$, $E=\Exc(\s)$, and $e:=\s|_{E}$. 
Let $i \colon C \hra \P^{5}$ and $j \colon E \hra \wt{\P}$ be the inclusions. 
Note that $|2H-E|$ is base point free if and only if $C$ is defined by quadratic equations. 
Consider the morphism 
$i \times \s \colon C \times \wt{\P} \to \P^{5} \times \P^{5}$. 
We denote the diagonal in $\P^{5} \times \P^{5}$ by $\Delta$ and its pull-back via $i \times \s$ by $\wt{E}$. 
Note that $\wt{E}$ is isomorphic to $E$. 
On $\P^{5} \times \P^{5}$, there is the following natural exact sequence 
$\Omega_{\P^{5}}(1) \boxtimes \mcO_{\P^{5}}(-1) \to \mcI_{\Delta/\P^{5} \times \P^{5}} \to 0$. 
Taking the pull-back under $i \times \s$, we have $\Omega_{\P^{5}}(1)|_{C} \boxtimes \mcO_{\wt{\P}}(-H) \to \mcI_{\wt{E}/C \times \wt{\P}} \to 0$. 
Since $\wt{E}$ is of codimension $2$ in $C \times \wt{\P}$, the kernel of the above surjection, say $\mcE$, is locally free. 
Thus, we obtain an exact sequence $0 \to \mcE \to \Omega_{\P^{5}}(1)|_{C} \boxtimes \mcO_{\Bl_{C}\P^{5}}(-H) \to \mcI_{\wt{E}/C \times \Bl_{C}\P^{5}} \to 0$ on $C \times \Bl_{C}\P^{5}$. 
For each closed point $x \in \Bl_{C}\P^{5}$, 
we set 
\begin{align}\label{def-Ex}
\mcE_{x}:=\mcE|_{C \times \{x\}}
\end{align}
and regard it as a locally free sheaf on $C$. 
In this setting, Mukai \cite{Mukai} proved the following theorem: 

\begin{thm}\label{thm-Mukai}\cite[Lemma~1 and Lemma~2]{Mukai}
Fix $x \in \Bl_{C}\P^{5}$. 
If $H^{1}(\mcE_{x} \otimes \xi)=0$ for a line bundle $\xi$ on $C$ of degree $2$, 
then the complete linear system $|2H-E|$ is free at $x$. 
Moreover, $\mcE_{x}$ is a rank $4$ subsheaf of $\Omega_{\P^{5}}(1)|_{C}$ such that $\det \mcE_{x}=\mcO_{C}(-1)$. 
\end{thm}

\subsection{Proof of Proposition~\ref{7-stability}}
The implication (1) $\ra$ (2) is obvious. 

We show the implication (2) $\ra$ (3). 
Let $\Phi \colon C \hra \P^{6}$ be the embedding given by $|\mcO_{C}(1)|$. 
Then there is a point $p \in \P^{6} \setminus \Phi(C)$ such that the image of $\Phi(C)$ by the projection from $p$ is $C \subset \P^{5}$. 
Hence we obtain the following exact sequence: 
\begin{align}\label{ex-proj-7}
0 \to \Omega_{\P^{5}}(1)|_{C} \to \Omega_{\P^{6}}(1)|_{C} \to \mcO_{C} \to 0.
\end{align}
Assume that $\Omega_{\P^{5}}(1)|_{C}$ has a rank $3$ subsheaf $\mcF$ such that $\mu(\mcF) \geq -7/5$. 
Since $\mcF$ is also a subsheaf of a slope stable bundle $\Omega_{\P^{6}}(1)|_{C}$ \cite[Theorem~1.3]{Brenner-Hein}, we have 
$-7/5 \leq \mu(\mcF) < -7/6$, which implies that $\deg \mcF=-4$. 
Let $\wh{\mcF} \subset \Omega_{\P^{5}}(1)|_{C}$ be the saturation of $\mcF$ in $\Omega_{\P^{5}}(1)|_{C}$. 
Since $\wh{\mcF}$ is also a rank $3$ subsheaf of $\Omega_{\P^{6}}(1)|_{C}$, 
we also have $\deg(\wh{\mcF})=-4$ by the same argument as above. 
Hence it follows that $\wh{\mcF}=\mcF$, i.e., $\mcF$ is saturated. 
Let $\mcG:=\Omega_{\P^{5}}(1)|_{C}/\mcF$, which is locally free. 
Considering the dual, we obtain an exact sequence 
$0 \to \mcG^{\vee} \to T_{\P^{5}}(-1)|_{C} \to \mcF^{\vee} \to 0$. 
We set $V=\C^{6}$ and $\P^{5}:=\P(V)$. 
Note that there is a surjection $V^{\vee} \otimes \mcO_{C} \epm T_{\P^{5}}(-1)|_{C}$. 
Let $I:=\Im(V^{\vee} \subset H^{0}(T_{\P^{5}}(-1)|_{C}) \to H^{0}(\mcF^{\vee}))$ and $K:=\Ker(V^{\vee} \to I)$. 
Then we have a surjection $I \otimes \mcO_{C} \epm \mcF^{\vee}$. 
Since $\rk \mcF^{\vee}=3$ and $H^{0}(\mcF^{\vee})=4$, 
we have $I=H^{0}(\mcF^{\vee})$. 
Thus $\Ker(I \otimes \mcO_{C} \to \mcF^{\vee})$ is an invertible sheaf $\mcO(-\eta)$, where $\deg \eta=4$. 
Taking the duals again, we have the following diagram: 
\[\xymatrix{
&&0&0&\\
&&\mcO_{C}(1) \ar[u] &\ar[l]\mcO_{C}(\eta)\ar[u]& \\
0 &K^{\vee} \otimes \mcO_{C} \ar[l] &V \otimes \mcO_{C}\ar[l] \ar[u]&I^{\vee} \otimes \mcO_{C}\ar[l] \ar[u]&\ar[l]0 \\
0 &\mcG \ar[l]&\Omega_{\P^{5}}(1)|_{C} \ar[l] \ar[u]& \mcF \ar[l] \ar[u]& \ar[l]0 \\
&&0\ar[u]&0. \ar[u]&
}\]
Let $l$ be the line corresponding to the injection $I^{\vee} \to V$. 
Under the natural morphism $\e \colon V \otimes \mcO_{\P(V)} \epm \mcO_{\P(V)}(1)$, 
the image of $I^{\vee} \otimes \mcO_{\P(V)}$ is nothing but $\mcI_{l/\P(V)}(1)$. 
Taking the restriction on $C$, we obtain surjections 
$I^{\vee} \otimes \mcO_{C} \epm \mcI_{l/\P(V)}(1)|_{C} \epm \mcI_{l \cap C/C} \otimes \mcO_{C}(1) \simeq \mcO_{C}(\eta)$. 
Hence $\mcI_{l \cap C/C}$ is isomorphic to $\mcO_{C}(-1) \otimes \mcO_{C}(\eta)$, which is of degree $-3$. 
Hence $l$ is a trisecant of $C$. 

Finally, we show the implication (3) $\ra$ (1). 
Assume that $\Omega_{\P^{5}}(1)|_{C}$ is slope stable for rank $3$ subsheaves. 
Let $x \in \Bl_{C}\P^{5}$ be a point and set $\mcE_{x}$ as in (\ref{def-Ex}). 
Then by Theorem~\ref{thm-Mukai} and the exact sequence (\ref{ex-proj-7}), we obtain inclusions $\mcE_{x} \hra \Omega_{\P^{5}}(1)|_{C} \hra \Omega_{\P^{6}}(1)|_{C}$. 
Note that the determinant bundles of the above three vector bundles are isomorphic. 

It is known (cf. \cite[Proposition~1.1]{Brenner-Hein}, \cite{Atiyah}) that there is a unique decomposition $\mcE_{x}:=\bigoplus_{i=1}^{N} \mcE_{i}$ such that each $\mcE_{i}$ is indecomposable. 
Since $\Omega_{\P^{6}}(1)|_{C}$ is slope stable and $\Omega_{\P^{5}}(1)|_{C}$ is slope stable for rank $3$ subsheaves, 
the following assertions hold for each $i$: 
\begin{itemize}
\item[(a)] $\mu(\mcE_{i})<-7/6$. 
\item[(b)] If $\rk \mcE_{i}=3$, then $\mu(\mcE_{i})<-7/5$. 
\end{itemize}
Note that $\sum_{i=1}^{N} \deg(\mcE_{i})=-7$. 
Set $r_{i}:=\rk(\mcE_{i})$ and $d_{i}:=\deg \mcE_{i}$. 
We may assume that $r_{1} \leq r_{2} \leq \cdots \leq r_{N}$. 
Then the sequence $(r_{i})$ is one of the following: 
$(1,1,1,1)$, $(1,1,2)$, $(1,3)$, $(2,2)$, or $(4)$. 
By Theorem~\ref{thm-Mukai}, 
it suffices to show that $H^{1}(\mcE_{x} \otimes \xi)=0$ for a general degree $2$ line bundle $\xi$ for each case. 
\begin{itemize}
\item[(i)] If $(r_{i})=(1,1,1,1)$, then we have $d_{i} < -7/6$ and hence $d_{i} \leq -2$, which contradicts that $\sum_{i=1}^{4}d_{i}=-7$. 
\item[(ii)] If $(r_{i})=(1,1,2)$, then we have $d_{1},d_{2} \leq -2$ and $d_{3} \leq -3$. 
Hence $d_{1}=d_{2}=-2$ and $d_{3}=-3$. 
Since $\mcE_{3}$ is slope stable, there is an exact sequence 
$0 \to \mcL_{-2} \to \mcE_{3} \to \mcL_{-1} \to 0$, 
where $\mcL_{i}$ is a line bundle of degree $i$. 
Let $\xi$ be a line bundle of degree $2$ such that $\xi^{-1}$ is neither $\mcE_{1}$, $\mcE_{2}$, nor $\mcL_{-2}$. 
Then we have $H^{1}(\mcE \otimes \xi)=
H^{1}(\mcE_{1} \otimes \xi) \oplus H^{1}(\mcE_{2} \otimes \xi) \oplus H^{1}(\mcE_{3} \otimes \xi)=0$. 
\item[(iii)] Assume that $(r_{i})=(1,3)$. 
By (a), we have $d_{1} \leq -2$.
By (b), we have $d_{2} \leq -5$. 
Hence we have $d_{1}=-2$ and $d_{2}=-5$. 
Let $\xi$ be a line bundle of degree $2$ such that $\xi^{-1} \not\simeq \mcE_{1}$. 
Then we have $H^{1}(\mcE_{x} \otimes \xi)
=H^{1}(\mcE_{1} \otimes \xi) \oplus H^{1}(\mcE_{2} \otimes \xi)
=H^{0}(\mcE_{1}^{\vee} \otimes \xi^{-1})^{\vee} \oplus \Hom(\xi,\mcE_{2}^{\vee})^{\vee}$. 
Note that $H^{0}(\mcE_{1}^{\vee} \otimes \xi^{-1})=0$ since $\xi^{-1} \not\simeq \mcE_{1}$. 
Since $\mcE_{2}$ is slope stable, so is $\mcE_{2}^{\vee}$. 
Since $\mu(\mcE_{2}^{\vee})=5/3$, we have $\Hom(\xi,\mcE_{2}^{\vee})=0$. 
\item[(iv)] If $(r_{i})=(2,2)$, then $d_{i} \leq -3$ for each $i$. 
Hence $\deg \mcE_{1}=-3$ and $\deg \mcE_{2}=-4$. 
Since $\mcE_{1}$ is slope stable and $\mcE_{2}$ is slope semi-stable, 
there are exact sequences 
$0 \to \mcL_{-2} \to \mcE_{1} \to \mcL_{-1} \to 0$
and
$0 \to \mcM_{-2} \to \mcE_{2} \to \mcM_{-2} \to 0$, 
where $\mcL_{i}$ is a line bundle of degree $i$ and $\mcM_{-2}$ is a line bundle of degree $-2$. 
Let $\xi$ be a line bundle of degree $2$ such that $\xi^{-1}$ is not equal to $\mcL_{-2}$ or $\mcM_{-2}$. 
Then It is easy to see $H^{1}(\mcE \otimes \xi)=0$. 
\item[(v)] If $\mcE_{x}$ is slope stable, then $H^{1}(\mcE_{x} \otimes \xi)=\Hom(\xi,\mcE_{x}^{\vee})^{\vee}=0$ for every line bundle $\xi$ of degree $2$. 
\end{itemize}

We complete the proof of Proposition~\ref{7-stability}. \qed

\subsection{Proof of Theorem~\ref{7-mainthm}}\label{subsec-Smoothing}

Let $X \subset \P^{5}$ be an arbitrary complete intersection of two hyperquadrics. 
We will construct an elliptic curve $C$ of degree $7$ on $X$ 
such that $C$ is defined by quadratic equations. 
We proceed with $4$ steps. 

\textit{Step 1.}
First of all, we confirm the following lemma. 
\begin{lem}\label{lem-finitelymanylines}
For an arbitrary closed point $x \in X$, 
lines on $X$ passing through $x$ are finitely many. 
\end{lem}
\begin{proof}
Let $\psi \colon \Bl_{x}X \to \P^{4}$ be the restriction of $\Bl_{x}\P^{5} \to \P^{4}$. 
Then the Stein factorization of $\psi$ is a crepant birational contraction. 
Since the proper transforms of the lines passing through $x$ are contracted by $\psi$, 
if there are infinitely many lines on $X$ passing through $x$, 
then $\psi$ contracts a divisor $D$, 
which contradicts the classification of weak Fano $3$-folds having crepant divisorial contractions \cite{JPR05}. 
\end{proof}

\textit{Step 2.}
In this step, we show the following lemma. 

\begin{lem}\label{lem-constCG}
Let $\G$ be a smooth conic on $X$. 
Then there exists a smooth elliptic curve $C$ of degree $5$
satisfying the following. 
\begin{enumerate}
\item[(i)] The scheme-theoretic intersection $C \cap \G$ is reduced one point. 
\item[(ii)] $C \cup \G$ has no trisecants on $X$. 
\end{enumerate}
\end{lem}

\begin{proof}
Let $\G$ be a smooth conic on $X$. 
Since the linear span $\braket{\G} \subset \P^{5}$ is not contained in $X$, 
we have $\braket{\G} \cap X=\G$. 
In particular, for any two points on $\G$, the line passing through them is not contained in $X$. 

Let $x_{1} \in \G$ be a point. 
Let $S_{0} \subset X$ be a general hyperplane section such that $S_{0}$ has no lines on $X$ passing through $x_{1}$. 
Set $\{x_{1},x_{2}\}=S_{0} \cap \G$. 
Let $l:=\braket{x_{1},x_{2}}$. 
Then $l \cap X=\{x_{1},x_{2}\}$. 
Note that $\Lambda:=|\mcO_{X}(1) \otimes \mf m_{x_{1},x_{2}}|$ is a 3-dimensional linear system whose general members are smooth. 
By Lemma~\ref{lem-finitelymanylines}, the lines passing through $x_{1}$ or $x_{2}$ on $X$ are finitely many. 
Hence 
\begin{align}\label{awayfromlines}
\text{a general member } S \in \Lambda \text{ does not contain any lines passing through } x_{1} \text{ or } x_{2}. 
\end{align}
We note that $S \cap \G=\{x_{1},x_{2}\}$ still holds. 

Let $\e \colon S \to \P^{2}$ be the blowing-up of $\P^{2}$ at $5$ points. 
Let $h$ be the pull-back of a line and $e_{1},\ldots,e_{5}$ the exceptional curves. 
Let $\s \colon \wt{S} \to S$ be the blowing-up at $x_{1}$ and 
$e_{0}$ the exceptional curve. 
Note that $\wt{S}$ is a del Pezzo surface by (\ref{awayfromlines}). 
We take a general member $\wt{C} \in |3h-(e_{0}+e_{1}+\cdots+e_{4})|$. 
Set $C=\s(\wt{C})$, which is smooth and passing through $x_{1}$. 
Note that $C$ is an elliptic curve of degree $5$ on $S$. 
By taking general $\wt{C}$, we may assume that $C$ does not pass through $x_{2}$. 
Since $S \cap \G=\{x_{1},x_{2}\}$, we have $C \cap \G =x_{1}$, which implies (i). 

We show (ii). 
First of all, there are no bisecants of $\G$ on $X$ since $\braket{\G} \cap X=\G$.
Moreover, $C$ has no trisecants on $X$. Indeed, since $S$ is a complete intersection of two hyperquadrics in $\P^{4}$, if $C$ has a trisecant $l$, then $l$ is contained in $S$. 
Since every line $l$ in $S$ is linear equivalent to $e_{i}$ for $1 \leq i \leq 5$, 
$h-(e_{i}+e_{j})$ for $1 \leq i < j \leq 5$, or
$2h-(e_{1}+\cdots+e_{5})$, 
we have $C.l \leq 2$. 

Hence, if there is a trisecant $l$ of $C \cup \G$, 
$l$ must be a bisecant of $C$. 
Note that every bisecant $l$ of $C$ is contained in $S$.
In fact, since $l=\braket{C \cap l} \subset \braket{C} \subset \braket{S}$, 
we have $l=l \cap X \subset \braket{S} \cap X=S$. 
Then by (\ref{awayfromlines}), we have $\emp = l \cap \{x_{1},x_{2}\} = l \cap \Gamma \cap S = l \cap \Gamma$. 
Therefore, every bisecant $l$ of $C$ does not meet $\Gamma$. 
Hence $\Gamma \cup C$ has no trisecants on $X$. 
We complete the proof of Lemma~\ref{lem-constCG}. 
\end{proof}

\textit{Step 3.}
Let $C$ and $\G$ be as in Lemma~\ref{lem-constCG}. 
We regard $C \cup \G$ as a reduced curve on $X$. 
Then the Hilbert polynomial of $C \cup \G$ is $7t$, 
where $t$ is a variable. 

In this step, we prove that $C \cup \G$ is strongly smoothable in $X$ in the sense of \cite{Hartshorne-Hirschowitz}.
Our proof is essentially the same as \cite[Proof of Lemma~6.2]{Beauville}. 
Let $x_{1}=C \cap \G$. 
By \cite[Theorem~4.1]{Hartshorne-Hirschowitz}, it suffices to check that 
$H^{1}(\mcN_{C/X}(-x_{1}))=0$ and 
$H^{1}(\mcN_{\G/X})=0$. 
Note that $H^{1}(\mcN_{C/X}(-x_{1}))=0$ is already proved in 
\cite[Proof of Lemma~6.2]{Beauville}. 
Let us show that $H^{1}(\mcN_{\G/X})=0$. 
Take a smooth hyperquadric $\Q^{4}$ in $\P^{5}$ containing $X$. 
Then $\mcN_{\G/\Q^{4}} \simeq \mcO_{\P^{1}}(2)^{3}$. 
From the normal bundle sequence 
$0 \to \mcN_{\G/X} \to \mcN_{\G/\Q^{4}} \simeq \mcO_{\P^{1}}(2)^{3} \to \mcN_{X/\Q^{4}}|_{\G} \simeq \mcO_{\P^{1}}(4) \to 0$, 
it follows that $\mcN_{\G/X}$ is spanned, which implies the vanishing $H^{1}(\mcN_{\G/X})=0$. 
Hence $C \cup \G$ is strongly smoothable. 


\textit{Step 4.} 
Finally, we prove Theorem~\ref{7-mainthm}. 
Let $C$ and $\G$ be as in Lemma~\ref{lem-constCG}. 
By Step 3, there is a smooth neighborhood $\Delta$ of $\Hilb_{7t}(X)$ at $[C \cup \G]$ such that the base change of the universal family $\mcC \to \D$ is a smoothing of $C \cup \G$. 
Let $C'$ be a general fiber of this smoothing. 
Since $C \cup \G$ is connected, so is $C'$. 
Hence $C'$ is an elliptic curve of degree $7$. 
Since $C \cup \G$ has no trisecants on $X$, general fibers have no trisecants on $X$ also. 
Hence $C'$ has no trisecants on $X$. 
Note that every trisecant $l$ of $C'$ in $\P^{5}$ must be contained in $X$ since $C'$ is on $X$ and $X$ is a complete intersection of two hyperquadrics in $\P^{5}$. 
Hence $C'$ has no trisecants on $\P^{5}$, which implies that $C'$ is defined by quadratic equations by Proposition~\ref{7-stability}. 
We complete the proof of Theorem~\ref{7-mainthm} and that of Theorem~\ref{thm-existence}. \qed

\subsection{An example of a non-weak Fano instanton bundle}

We conclude this paper as showing that 
every complete intersection of two hyperquadrics $X \subset \P^{5}$ has an instanton bundle $\mcE$ with $c_{2}(\mcE)=3$ 
such that $\mcE$ is not a weak Fano bundle. 
To show the above result, we show the following proposition. 
\begin{prop}\label{prop-trisecex}
Let $X \subset \P^{5}$ be an arbitrary smooth complete intersection of two hyperquadrics. 
Let $l \subset X$ be a general line such that $\mcN_{l/X} \simeq \mcO_{l}^{\oplus 2}$. 
Then there exists a non-degenerate smooth elliptic curve $C \subset X \subset \P^{5}$ of degree $7$
such that $C$ transversally meets $l$ at $3$ points. 
\end{prop}
\begin{proof}
Let $\sigma \colon \wt{X}:=\Bl_{l}X \to X$ be the blowing-up. 
Then the restriction of the projection morphism $\Bl_{l}\P^{5} \to \P^{3}$ is 
a birational morphism $\tau \colon \wt{X} \to \P^{3}$. 
It is known that $\tau$ is the blowing-up along a smooth curve $B \subset \P^{3}$ of genus $2$ and degree $5$ (cf. \cite[Proposition~3.4.1]{Fanobook}). 
Letting $E=\Exc(\sigma)$ and $Q_{0}:=\tau(E)$, 
we can show that $Q_{0} \subset \P^{3}$ is a smooth quadric surface containing $B$. 
\begin{align}\label{diag-trisecex}
\xymatrix{
&&&E \ar[lldd]_{\simeq} \ar[rrdd]^{\s|_{E}} \ar@{^{(}->}[d]&& \\
&&&\ar[ld]_{\tau}\Bl_{B}\P^{3}=\wt{X}=\Bl_{l}X\ar[rd]^{\sigma} \ar[rd]&& \\
B&\ar@{}[l]|{\subset}Q_{0}&\ar@{}[l]|{\subset}\P^{3}&&X\ar@{}[r]|{\supset}&l
}
\end{align}
From now on, we fix an isomorphism 
\begin{align}\label{isom-trisecex}
Q_{0} \simeq (\P^{1})^{2} \text{ s.t. } B \in |\mcO_{(\P^{1})^{2}}(2,3)|.
\end{align}
Then $\pr_{1} \colon (\P^{1})^{2} \to \P^{1}$ coincides with $\s|_{E} \colon E \to l$ under this identification $E \simeq Q_{0} \simeq (\P^{1})^{2}$. 

Let us take a general quadric $Q_{1} \subset \P^{3}$. 
Then $C_{01}:=Q_{0} \cap Q_{1}$ is a smooth elliptic curve and 
$B \cap Q_{1}$ consists of $10$ points. 
Note that $C_{01}$ contains $B \cap Q_{1}$. 
Take distinct $5$ points $p_{1},\ldots,p_{5} \in B \cap Q_{1}$
 and set $Z=\{p_{1},\ldots,p_{5}\}$. 
\begin{claim}\label{claim-trisecex}
The following assertions hold. 
\begin{enumerate}
\item The linear system $|\mcI_{Z/Q_{1}}(2)|$ has smooth members and $\Bs|\mcI_{Z/Q_{1}}(2)|=Z$. 
\item A general smooth member $C_{12} \in |\mcI_{Z/Q_{1}}(2)|$ satisfies the following conditions. 
\begin{enumerate}
\item $C_{12}$ meets $B$ transversally and at $5$ points $p_{1},\ldots,p_{5}$. 
\item $C_{12}$ meets $Q_{0}$ transversally at $8$ points $p_{1},\ldots,p_{5},q_{1},q_{2},q_{3}$. 
\item Under the identification $Q_{0} \simeq (\P^{1})^{2}$ as in (\ref{isom-trisecex}), 
we have $\pr_{1}(q_{i}) \neq \pr_{1}(q_{j})$ for $1 \leq i<j \leq 3$. 
\end{enumerate}
\end{enumerate}
\end{claim}
\begin{proof}
We show (1). 
Since $C_{01}$ contains $Z$, we have $C_{01} \in |\mcI_{Z/Q_{1}}(2)|$. 
Hence general members of $|\mcI_{Z/Q_{1}}(2)|$ are smooth. 
Moreover, the proper transform of $C_{01}$ on $\Bl_{Z}Q_{1}$ is an anti-canonical member, which implies that $|-K_{\Bl_{Z}Q_{1}}|$ is base point free. 
Hence $\Bs|\mcI_{Z/Q_{1}}(2)|=Z$. 

We show (2). 
The condition (a) is general since $B \cap Q_{1}=\{p_{1},\ldots,p_{10}\}$ and 
$\Bs |\mcI_{Z/Q_{1}}(2)| = Z$. 
In order to show the conditions (b) and (c) are general, 
we recall the inclusions $Z  \subset C_{01}=Q_{0} \cap Q_{1} \subset Q_{1}$ and consider the following exact sequence:
\[0 \to \mcO_{Q_{1}} \to \mcI_{Z/Q_{1}}(2) \to \mcI_{Z/C_{01}}(2) \to 0.\]
It follows from the above exact sequence that 
the restriction morphism 
$H^{0}(\mcI_{Z/Q_{1}}(2)) \to H^{0}(\mcI_{Z/C_{01}}(2))$
is surjective. 
Note that $\mcL_{3}:=\mcI_{Z/C_{01}}(2)$ is an invertible sheaf of degree $3$ on the elliptic curve $C_{01}$. 
Hence a general member $Z_{3} \in |\mcI_{Z/C_{01}}(2)|$ consists of $3$ points $\{q_{1},q_{2},q_{3}\}$. 
Therefore, the condition (b) is general. 

Finally, we show the condition (c) is general. 
Recall the inclusion $C_{01} \subset Q_{0} \simeq (\P^{1})^{2}$. 
Assume that $\#(\pr_{1}(Z_{3}))_{\red} \leq 2$ for a general member $Z_{3} \in |\mcI_{Z/C_{01}}(2)|=|\mcL_{3}|$. 
This means that a general section $s \in H^{0}(C_{01},\mcL_{3})$ can be divided by a section of $\mcL_{2}:=\mcO_{\P^{1}}(1,0)|_{C_{01}}$, which is a line bundle of degree $2$. 
Then the linear map $H^{0}(\mcL_{2}) \otimes H^{0}(\mcL_{3} \otimes \mcL_{2}^{\vee}) \to H^{0}(\mcL_{3})$ is surjective, which is a contradiction. 
Hence $\#(\pr_{1}(Z_{3}))=3$ for general member $Z_{3}=\{q_{1},q_{2},q_{3}\}$. 
Therefore, the condition (c) is general. 
We complete the proof of Claim~\ref{claim-trisecex}. 
\end{proof}

Let $C_{12} \in |\mcI_{Z/Q_{1}}(2)|$ be a general smooth member. 
Recall the birational map $\P^{3} \dashrightarrow X$  in the diagram (\ref{diag-trisecex}). 
We show that the proper transform $\ol{C_{12}} \subset X$ of $C_{12}$ is a non-degenerate septic smooth elliptic curve on $X$ which meets the line $l$ at $3$ points transversally. 

First, we show $\ol{C_{12}} \subset X$ is of degree $7$. 
Let $\wt{C_{12}} \subset \Bl_{B}\P^{3}$ be the proper transform of $C_{12}$ on the blowing-up $\tau \colon \Bl_{B}\P^{3} \to \P^{3}$. 
By the condition Claim~\ref{claim-trisecex}~(2)~(a), 
we have $\Exc(\tau).\wt{C_{12}}=5$. 
Since 
$\sigma^{\ast}\mcO_{X}(1) \simeq \tau^{\ast}\mcO_{\P^{3}}(3) \otimes \mcO(-\Exc(\tau))$, 
we have $\sigma^{\ast}\mcO_{X}(1).\wt{C_{12}}=7$. 
Hence $\ol{C_{12}}$ is of degree $7$.

Next, we show $\ol{C_{12}}$ is a smooth elliptic curve meeting $l$ transversally at $3$ points. 
By the condition Claim~\ref{claim-trisecex}~(2)~(b), 
the proper transform $E \subset \Bl_{B}\P^{3}$ of $Q_{0}$ meets $\wt{C_{12}}$ transversally at $3$ points. 
As we saw in the diagram (\ref{diag-trisecex}), 
$E \simeq Q_{0}$ is contracted by $\sigma \colon \Bl_{B}\P^{3} \to X$ and $\sigma|_{E} \colon E \to l$ corresponds to the 1st projection $\pr_{1} \colon (\P^{1})^{2} \to \P^{1}$ under the isomorphism (\ref{isom-trisecex}). 
Then by the condition Claim~\ref{claim-trisecex}~(2)~(c), 
$\s|_{\wt{C_{12}}} \colon \wt{C_{12}} \to \ol{C_{12}}$ is isomorphic. 
Moreover, $\ol{C_{12}}$ meets $l$ at the image of the $3$ points $q_{1},q_{2},q_{3}$. 
Therefore, $\ol{C_{12}}$ is a smooth elliptic curve of degree $7$ which meets $l$ at $3$ points. 

Finally, we confirm that the embedding $\ol{C_{12}} \subset X \subset \P^{5}$ is non-degenerated. 
Assume that there is a hyperplane section $H$ of $X$ containing $\ol{C_{12}}$. 
Since $l$ is a trisecant of $\ol{C_{12}}$, $H$ must contain $l$. 
Thus the proper transform $\ol{H}$ of $H$ on $\P^{3}$ is a hyperplane of $\P^{3}$. 
Then $\ol{H}$ contains $C_{12}$, which contradicts that $C_{12}$ is of degree $4$ in $\P^{3}$. 
We complete the proof of Proposition~\ref{prop-trisecex}. 
\end{proof}

\begin{rem}\label{rem-nonweakfano}
Let $X \subset \P^{5}$ be an arbitrary smooth complete intersection of two hyperquadrics. 
By Proposition~\ref{prop-trisecex}, there is a smooth elliptic curve $C$ of degree $7$ having a trisecant. 
By the Hartshorne-Serre construction, 
there is a rank $2$ vector bundle $\mcE$ fitting into $0 \to \mcO_{X}(-1) \to \mcE \to \mcI_{C/X}(1) \to 0$. 
This vector bundle $\mcE$ is an instanton since $H^{1}(\mcE(-1))=H^{1}(\mcI_{C/X})=0$ and $c_{2}(\mcE)=3$. 
However, $\mcE$ is not a weak Fano bundle since there are surjections $\mcE|_{l} \epm \mcI_{C/X}(1) \epm \mcO_{l}(-2)$. 
\end{rem}

\providecommand{\bysame}{\leavevmode\hbox to3em{\hrulefill}\thinspace}
\providecommand{\MR}{\relax\ifhmode\unskip\space\fi MR }
\providecommand{\MRhref}[2]{%
  \href{http://www.ams.org/mathscinet-getitem?mr=#1}{#2}
}
\providecommand{\href}[2]{#2}

\end{document}